\documentclass[12pt]{amsart}
\usepackage{amsmath,amsthm,amsfonts,amssymb,latexsym}
\theoremstyle{plain}

\newtheorem{thm}{Theorem}[section]

\newtheorem{lem}[thm]{Lemma}
\newtheorem{pro}[thm]{Proposition}
\newtheorem{cor}[thm]{Corollary}

\newtheorem{exa}[thm]{Example}

\newcommand{\Ker}{\operatorname{Ker}}

\newcommand{\Irr}{\operatorname{Irr}}
\newcommand{\SL}{\operatorname{SL}}
\newcommand{\FF}{\mathbb{F}}

\newcommand{\tV}{\tilde{V}}
\newcommand{\tU}{\tilde{U}}
\newcommand{\bG}{\bar{G}}

\newcommand{\cn}{\operatorname{cn}}

    \def \mod#1{\, {\rm mod} \, #1 \, }

\newcommand{\tw}[1]{{}^#1\!}

\newcommand{\diag}{{\mathrm {diag}}}

\newcommand{\codim}{{\mathrm {codim}}}

\newcommand{\Res}{{\mathrm {Res}}}

\newcommand{\AAA}{{\sf A}}

\newcommand{\QQ}{{\mathbb Q}}

\newcommand{\ZZ}{{\mathbb Z}}
\newcommand{\GC}{{\mathcal G}}
\newcommand{\CL}{{\mathcal C}}
\newcommand{\TC}{{\mathcal T}}
\newcommand{\EC}{{\mathcal E}}

\newcommand{\GCF}{\GC^{F}}

\newcommand{\eps}{\epsilon}
\newcommand{\al}{\alpha}

\newcommand{\gam}{\gamma}
\newcommand{\lam}{\lambda}
\newcommand{\om}{\omega}

\newcommand{\supp}{{\mathrm {supp}}}
\newcommand{\ppd}{{\mathrm {ppd}}}
\newcommand{\St}{{\sf {St}}}

\renewcommand{\mod}{\bmod \,}

\begin{document}

\title{Effective Results on the Waring Problem for Finite Simple Groups}
\author{Robert M. Guralnick}
\address{Department of Mathematics, University of Southern California,
Los Angeles, CA 90089-2532, USA}
\email{guralnic@usc.edu}
\author{Pham Huu Tiep}
\address{Department of Mathematics, University of Arizona, Tucson,
AZ 85721-0089, USA}
\email{tiep@math.arizona.edu}



\thanks{The authors are grateful to Gunter Malle for helpful comments on the paper,
in particular for a suggestion that improves part of the proof of Proposition
\ref{values}.}

\thanks{The authors were partially supported by the NSF 
(grants DMS-1001962, DMS-0901241, and DMS-1201374). 
The first author was also partially supported by  the Simons Foundation Fellowship 224965}

\begin{abstract}
Let $G$ be a finite quasisimple group of Lie type. We show that there are 
regular semisimple elements $x,y \in G$, $x$ of prime order, and $|y|$ is 
divisible by at most two primes, such that 
$x^G \cdot y^G \supseteq G \setminus Z(G)$. In fact in all but four cases,
$y$ can be chosen to be of square-free order. Using this result, we prove
an effective version of one of the main results of \cite{LST1} by showing
that, given any integer $m \geq 1$, if the order of a finite simple
group $S$ is at least $m^{8m^{2}}$, then every element in $S$ is a product of two 
$m^{\mathrm {th}}$ powers. Furthermore, the verbal width of $x^m$ on any finite simple
group $S$ is at most $80m \sqrt{2\log_2 m}+56$.
We also show that, given any two non-trivial words
$w_1$, $w_2$, if $G$ is a finite quasisimple group of large enough order,
then $w_1(G)w_2(G) \supseteq G \setminus Z(G)$.   
\end{abstract}

\maketitle

\section{Introduction}

Let $G$ be a finite non-abelian simple group, or more generally, a finite
quasisimple group (that is, $G = [G,G]$ and $G/Z(G)$ is simple). Recently, 
various problems involving $G$, such as Waring-type problems and generation
problems, cf. for instance \cite{MSW}, \cite{LST1}, \cite{GM},  
have been resolved, crucially relying on the fact that every non-central element
of $G$ is a product of conjugates of two specific elements in $G$.   
Building on earlier work of \cite{MSW}, \cite{LST1}, and \cite{GM}, we prove
the following refinement of these results on covering non-central elements in 
finite quasisimple groups of Lie type:

\begin{thm}\label{main1}
{\sl Let $\GC$ be a simple simply connected algebraic group in characteristic 
$p > 0$ and let $F~:~\GC \to \GC$ be a generalized Frobenius endomorphism such that 
$G := \GC^F$ is quasisimple. Then there exist (not necessarily distinct)
primes $r,s_1,s_2 \neq p$ and 
regular semisimple elements $x,y \in G$ such that $|x| = r$, 
$y$ is an $\{s_1,s_2\}$-element, and $x^G\cdot y^G \supseteq G \setminus Z(G)$.
In fact $s_1 = s_2$ unless $\GC$ is of type $B_{2n}$ or $C_{2n}$.
Moreover, if 
$$G \notin \{ SL_2(5), SL_2(17), Sp_4(3), Spin_9(3)\}$$ 
then the order of $y$ can also be chosen to be square-free.}
\end{thm} 

\begin{cor}\label{simple1}
{\sl Let $\GC$ be a simple simply connected algebraic group in positive 
characteristic and let $F~:~\GC \to \GC$ be a generalized Frobenius endomorphism such that 
$G := \GC^F$ is quasisimple. Assume in addition that 
$$G \notin \{ SL_2(5), SL_2(17), Sp_4(3)\}.$$ 
Then there exist (not necessarily distinct) 
primes $r,s_1,s_2$ with the following properties:

{\rm (i)} There are elements $x,y \in G$ of 
square-free order such that $x$ is regular semisimple of order $r$, 
$y$ is an $\{s_1,s_2\}$-element, and 
$x^G\cdot y^G \supseteq G \setminus Z(G)$.
In fact $s_1 = s_2$ unless $\GC$ is of type $B_{2n}$ or $C_{2n}$. Furthermore,
$y$ can also be chosen to be regular semisimple unless possibly 
$G \cong Spin_9(3)$.

{\rm (ii)} For any elements $a,b,c \in G$ of order $r$,
$a^G \cdot b^G \cdot c^G = G$.}
\end{cor}

\begin{cor}\label{simple2}
{\sl Let $S$ be a finite non-abelian simple group.
Then there exist (not necessarily distinct)
primes $r,s_1,s_2$ and elements $x,y \in S$ such that 

{\rm (i)} $|x| = r$,

{\rm (ii)} either $|y| = s_1$, or $S \in \{ PSp_{4n}(q), \Omega_{4n+1}(q)\}$, 
$s_1 \neq s_2$ and $|y| = s_1s_2$,

{\rm (iii)} $x^S\cdot y^S \supseteq S \setminus \{1\}$, 
$x^S \cdot x^S \cdot x^S = S$.\\
Moreover, if $S$ is of Lie-type then $x$ and $y$ can be chosen to
be regular semisimple.}
\end{cor}

Note that a (slightly weaker) 
version of Theorem \ref{main1} also holds for $r=s=p$: every non-central
element of $G$ is a product of two unipotent elements, cf. 
\cite[Corollary, p. 3661]{EG}. Furthermore, in a sense Corollary \ref{simple2}
yields another approximation towards Thompson's conjecture (which states that every 
finite non-abelian simple group $S$ possesses a conjugacy class $C$ such that
$C^2 = S$). We also note that an asymptotic version of Corollary 
\ref{simple2}(iii) was established in \cite[Corollary 2.3]{Sh}: {\it Every large 
enough finite simple group $S$ has a conjugacy class $C$ such that $C^3 = S$}.  
 
\smallskip
%We expect Theorem \ref{main1} to be useful in various situations. As an application,
Theorem \ref{main1} allows us to prove the following effective version of the main 
result of \cite{LST1} for the Waring problem in the case of powers:

\begin{thm}\label{main2}
{\sl Let $k, l \geq 1$ be any two integers and let $m := \max(k,l)$. If $S$ 
is any finite simple group of order at least $m^{8m^2}$, then every element in 
$S$ can be written as $x^k\cdot y^l$ for some $x,y \in S$.}
\end{thm}

The main result of \cite{LST1} implies that the {\it width} of the word
$w(x) = x^m$ on any finite non-abelian simple group $S$ is $2$ 
(that is, every element of $S$ is a product of two values of $w$ on $S$),
if $|S|$ is sufficiently large (but no explicit bound is given). 
Theorem \ref{main2} shows in particular that the width of 
$w(x)=x^m$ on any finite simple group $S$ is $2$ if $|S| \geq m^{8m^2}$.  

Without {\it any condition} on $|S|$, Theorem \ref{main2} becomes false -- 
there are various examples, cf. \S3, showing 
that the width of $x^m$ can grow unbounded
even on simple groups $S$ containing non-trivial $m^{\mathrm {th}}$ powers.
However, the width of $x^m$ on any finite simple group $S$ is bounded 
universally, say by $70$, see Corollary \ref{cor:powers}, as long as 
there is a prime $p$ that divides $|S|$ but not $m$. More generally, we prove

\begin{cor}\label{power}
{\sl Let $m \geq 1$ be any integer and let $S$ be any finite simple group 
such that $m$ is not divisible by $\exp(S)$. Then any element of $S$ is a 
product of at most 
$$f(m):= 80m\sqrt{2\log_2 m} + 56$$ 
$m^{\mathrm {th}}$ powers in $S$.}
\end{cor}

Corollary \ref{power} implies that, for any $m \geq 1$, 
the {\it verbal width} of the word $x^m$
on any finite simple group $S$ is at most $f(m)$ (i.e.
any element of the subgroup $\langle g^m \mid g \in S\rangle$ is a product of
at most $f(m)$ $m^{\mathrm {th}}$ powers in $S$). Thus 
Corollaries \ref{power} and \ref{cor:powers} yield effective versions of the 
main results of \cite{MZ} and \cite{SW}. For arbitrary finite groups,
the verbal width of the word $x^m$ on any $d$-generated finite group is 
bounded universally by an (implicit) function of $m$ and $d$, 
see \cite[Theorem 1]{NS}.   
   
\smallskip
For an arbitrary word $w \neq 1$, the main result of \cite{LST2} shows
that the width of $w$ on any finite quasisimple group $G$ is at most $3$
(that is, every element in $G$ is a product of at most $3$ values of $w$ on
$G$), if $|G|$ is sufficiently large. It remained an open question whether 
every {\it non-central} element of $G$ is a product of at most $2$ values of
$w$ on $G$. Our next result answers this question in the affirmative:   

\begin{thm}\label{main3}
{\sl {\rm (i)} Let $w \in F_d$ be a non-trivial word in the free group on $d$
generators. Then there exists a constant $N = N_w$ depending on $w$
such that for all finite quasisimple groups $G$ of order greater than $N$
we have $w(G)^2 \supseteq G \setminus Z(G)$.

{\rm (ii)} Let $w_1,w_2 \in F_d$ be two non-trivial words in the free group on $d$
generators. Then there exists a constant $N = N_{w_1,w_2}$ depending on $w_1$ and 
$w_2$ such that for all finite quasisimple groups $G$ of order greater than $N$
we have $w_1(G)w_2(G) \supseteq G \setminus Z(G)$.}
\end{thm}

As shown in \cite[Corollary 4.3]{LST2}, central elements are real obstructions 
for $w(G)^2$, respectively $w_1(G)w_2(G)$, to coincide with $G$. Furthermore,
there are many non-trivial words $w$ (for instance $w(x) = x^2$) which are 
not surjective on any finite quasisimple group. So in this 
sense, Theorem \ref{main3} is best possible for finite quasisimple groups.

\smallskip
The paper is organized as follows. First we prove Theorems \ref{main1} and 
Corollaries \ref{simple1} and \ref{simple2} in \S2. In \S3 we prove 
Theorem \ref{main2}, Corollary \ref{power}  
and further results on the Waring problem for powers. Finally, Theorem \ref{main3}
is established in \S4.
  
\section{Covering non-central elements in quasisimple groups of Lie type}
This section is devoted to prove Theorem \ref{main1} and Corollaries \ref{simple1},
\ref{simple2}. Keep the notation of
the theorem. Note that if $x,y \in G$ are regular 
semisimple, then a result essentially proved by Gow \cite{Gow}, cf. 
\cite[Lemma 5.1]{GT}, shows that $x^G \cdot y^G$ contains every non-central 
semisimple element of $G$. 
%On the other hand, since in all cases we choose 
%$r \neq s$ with at least one of $r,s$ being coprime to $|Z(G)|$, $x^G \cdot y^G$
%for any $r$-element $x \neq 1$ and any $s$-element $y \neq 1$ cannot contain any
%central element of $G$. So to prove $x^G \cdot y^G = G \setminus Z(G)$,
We also record the following observation:

\begin{lem}\label{23}
{\sl In the notation of Theorem \ref{main1}, let $r$ be a prime with 
the following properties:

{\rm (i)} Any element of order $r$ in $G$ is regular semisimple.

{\rm (ii)} For any $x \in G$ of order $r$, there exists a regular semisimple 
element $y \in G$ such that $x^G \cdot y^G \supseteq G \setminus Z(G)$.\\
Then $a^G \cdot b^G \cdot c^G = G$ for any elements $a,b,c \in G$ of order $r$.}
\end{lem}

\begin{proof}
We apply (ii) to $x=a$.
As mentioned above, $y \in b^G \cdot c^G$ since $b$ and $c$ are both 
regular semisimple (and $y$ is certainly non-central semisimple). Hence, 
$$G \setminus Z(G) \subseteq a^G \cdot y^G \subseteq a^G \cdot b^G \cdot c^G.$$
On the other hand, if $z \in Z(G)$, then $zc^{-1}$ is non-central semisimple
and so $zc^{-1} \in a^G \cdot b^G$, whence
$z \in a^G \cdot b^G \cdot c^G$.   
\end{proof}

In what follows, we will choose $r$ to satisfy the condition (i) of Lemma
\ref{23}. Hence, fixing any $x \in G$ of order $r$ and choosing $y$ 
suitably, it suffices to show that $g \in x^G\cdot y^G$, equivalently,
\begin{equation}\label{count}
  \sum_{\chi \in \Irr(G)}\frac{\chi(x)\chi(y)\overline{\chi}(g)}{\chi(1)} 
    > 0
\end{equation}
for all non-central non-semisimple elements $g \in G$.   

\subsection{Type $D_n$ with $2|n \geq 4$}
Throughout this subsection, let $\GC^F = G = Spin^+_{2n}(q)$ with $2|n \geq 4$. 
In this case, 
it is already proved in \cite[Theorem 1.1.4]{LST1} and \cite[Theorem 7.6]{GM}
that $G$ possesses two regular semisimple elements $y_1$, $y_2$ such that
$y_1^G \cdot y_2^G$ contains $G \setminus Z(G)$. But the order of one of these
two elements is {\it not} a prime power; moreover, the pair of maximal tori 
containing these elements does not work well in further applications that we have 
in mind, including Theorems \ref{main2} and \ref{main3}.

Note that Theorem \ref{main1} and Corollaries \ref{simple1}, \ref{simple2} hold for 
$G = \Omega^+_8(2)$ by choosing $x$, $y \in G$ of order $7$ (as one can check
using \cite{GAP}). In what follows we will therefore assume that $(n,q) \neq (4,2)$.
Following the approach of \cite[\S2]{LST1}, we consider some $F$-stable maximal
tori $\TC_1$, $\TC_2$ of $\GC$ such that $T_1 := \TC_1^F$ is of type 
$T^{+,+}_{n-1,1}$ (so it has order $(q^{n-1}-1)(q-1)$) and $T_2 := \TC_2^F$ is of type 
$T^{-,-}_{n-1,1}$ (so it has order $(q^{n-1}+1)(q+1)$).

\begin{lem}\label{tori}
{\sl Suppose $2|n \geq 4$. Then the tori $T_1$ and $T_2$ are weakly 
orthogonal in the sense of \cite[Definition 2.2.1]{LST1}.}
\end{lem}

\begin{proof}
We follow the proof of \cite[Proposition 2.6.1]{LST1}.
Here, the dual group $G^{*}$ is $PCO(V)^{\circ}$, where
$V = \FF_{q}^{2n}$ is endowed with a suitable quadratic form $Q$; see
\cite[Lemma 7.4]{TZ1} for an explicit description of the groups $G^{*}$ and
$H := CO(V)^{\circ}$. Consider the
complete inverse images in $H$ of the tori dual to
$T_1$ and $T_2$, and assume $g$ is an element
belonging to both of them. We need to show that $g \in Z(H)$.
To this end, consider the spectrum $S$ of the semisimple element $g$ on $V$ as a
multiset. Let $\gam \in \FF_{q}^{\times}$ be the {\it conformal coefficient} of
$g$, i.e. $Q(g(v)) = \gam Q(v)$ for all $v \in V$. Then $S$ can be represented
as the joins of multisets $X \sqcup Y$ and $Z \sqcup T$, where
$$\begin{array}{ll}
  X := \{x,x^{q}, \ldots,x^{q^{n-2}},\gam x^{-1},\gam x^{-q}, \ldots,
             \gam x^{-q^{n-2}}\},& Y := \{y,\gam y^{-1}\},\\
  Z := \{z,z^{q}, \ldots,z^{q^{n-2}},\gam z^{-1},\gam z^{-q}, \ldots,
             \gam z^{-q^{n-2}}\}, & T := \{t,\gam t^{-1}\},\end{array}$$
for some $x,y,z,t \in \overline{\FF}_{q}^{\times}$; furthermore,
$x^{q^{n-1}-1} = 1 = y^{q-1}$ and $z^{q^{n-1}+1} = \gam = t^{q+1}$.

Since $|S| = 2n < |X|+|Z| = 4n-4$, we may assume that $x \in X \cap Z$. It follows
that $x^2 = \gam$ and so $x^{2(q-1)} = 1$. As $n$ is even, we see that 
$|x|$ divides $\gcd(2(q-1),q^{n-1}-1) = q-1$, i.e. $x \in \FF_q^{\times}$. Thus
$X = Z = \underbrace{\{x,x, \ldots,x\}}_{2n-2}$ (as multisets). In turn, this 
forces $Y \cap T \neq \emptyset$ and so we may assume that $y \in Y \cap T$. 
Arguing as with $x$, we get $y \in \FF_q^{\times}$, $y^2 = \gam$, and 
$Y = T = \{y,y\}$. Recall that 
$x^2 = \gam$. Now if $x = y$, then $S = \underbrace{\{x,x, \ldots,x\}}_{2n}$
and so $g \in Z(H)$ as $g$ is semisimple. 

Assume that $x \neq y$, whence $q$ is odd and $y = -x$. Using the decomposition
$S = X \sqcup Y$, we see that $V$ is the orthogonal sum $V_1 \oplus V_2$,
where $V_1 = \Ker(g-x \cdot 1_V)$ and  $V_2 = \Ker(g+x \cdot 1_V)$. Moreover, since
$T_1$ has type $T^{+,+}_{n-1,1}$, $V_1$ and $V_2$ are both of type $+$. But then
the same argument applied to the decomposition $S = Z \sqcup T$ and the torus $T_2$
implies that $V_1$ and $V_2$ must be both of type $-$, a contradiction.
\end{proof}  

By \cite{Zs}, since $n-1 \geq 3$ is odd, $q^{n-1}-1$ has a 
{\it primitive prime divisor} $r$, i.e. $r$ divides $q^{n-1}-1$ but not
$\prod^{n-2}_{i=1}(q^i-1)$. In what follows, we will let 
$\ppd(q,n)$ denote any such divisor. Similarly, we take 
$s = \ppd(q,2n-2)$ (which exists since we are assuming $(n,q) \neq (4,2)$).
%in which case we set $s = 3$. 
Arguing as in the proof of \cite[Lemma 2.4]{MT}, we can show that any element 
$x \in G$ of order $r$ is regular semisimple, and certainly we can choose
$x \in T_1$. Similarly, we can find a regular semisimple element 
$y \in T_2$ of order $s$. 
%(note that $|y| = 9$ in the exceptional case $(n,q) = (4,2)$). 
In fact, if $(n,q) \neq (4,4)$, then, writing $q = p^f$, we can choose
$r = \ppd(p,(n-1)f)$ and $s = \ppd(p,2(n-1)f)$, which ensures that
$r > (n-1)f$ and $s > 2(n-1)f$.    

With the above choice of $(x,y)$, we prove the following key statement:

\begin{pro}\label{values}
{\sl There exist precisely four irreducible characters of $G$ which are nonzero
on both $x$ and $y$: the principal character $1_G$, the Steinberg character 
$\St$, and two more unipotent characters $\al$ and $\beta$ of degree 
$$\al(1) = \frac{(q^n-1)(q^{n-1}+q)}{q^2-1},~~~
  \beta(1) = \frac{q^{n^2-3n+2}(q^n-1)(q^{n-1}+q)}{q^2-1}.$$
All of them take values $\pm 1$ at $x$ and at $y$.}
\end{pro}
  
\begin{proof}
1) Consider any $\chi \in \Irr(G)$ with 
$\chi(x)\chi(y) \neq 0$. Since $T_1$ and $T_2$ are weakly orthogonal by Lemma 
\ref{tori}, $\chi$ must be unipotent by \cite[Proposition 2.2.2]{LST1}. 
Now for $n=4$ the statement follows by inspecting the values of the unipotent 
characters of $G$ as given in {\sf {Chevie}} \cite{Chevie}. From now on
we will assume $n \geq 6$.

To identify $\chi$ among the unipotent characters of $G$, 
one could follow the proof of Propositions 3.3.1 and
7.1.1 of \cite{LST1}, but instead we will use the hook-cohook approach of 
\cite[\S3.3]{LMT}. Let $\chi$ correspond to the \emph{symbol}
$S=(X,Y)$ which is a pair of strictly increasing sequences $X=(x_1<\ldots<x_k)$,
$Y=(y_1<\ldots<y_l)$ of non-negative integers, with $0 \notin X \cap Y$,  
\begin{equation}\label{rk1}
  n = \sum_{i=1}^k x_i+\sum_{j=1}^l y_j
    -\frac{(k+l)(k+l-2)}{4},
\end{equation}
and $4|(k-l)$. (Such a symbol corresponds to two unipotent characters of $G$
if $X = Y$.) A \emph{hook of $S$}
is a pair $(b,c)\in\ZZ^2$ with $0 \leq b<c$ and either $b\notin X$, $c\in X$,
or $b\notin Y$, $c\in Y$. A \emph{cohook of $S$} is a pair $(b,c)\in\ZZ^2$ with 
$0 \leq b<c$ and either $b\notin Y$,
$c\in X$, or $b\notin X$, $c\in Y$. We also set
$$a(S):=\sum_{\{b,c\}\subseteq S}\min\{b,c\}-\sum_{i\ge1}\binom{k+l-2i}{2},$$
where the first sum runs over all 2-element subsets of the multiset $X\cup Y$
of entries of $S$, and $b(S)=\lfloor|S-1|/2\rfloor-|X\cap Y|$ if $X\ne Y$,
respectively $b(S)=0$ else. Then
\begin{equation}\label{deg}
  \chi(1)=q^{a(S)}\frac{|G|_{q'}}{2^{b(S)}\prod_{(b,c)\text{ hook}}(q^{c-b}-1)
            \prod_{(b,c)\text{ cohook}}(q^{c-b}+1)},
\end{equation}
where the products run over hooks, respectively cohooks of $S$, cf. 
\cite[Bem.~3.12 and~6.8]{M2}.

\smallskip
2) First we use (\ref{rk1}) to bound $x_k$ and $y_l$ in terms of $n$. 
Recall that $4|t := k-l$. If $x_1 = 0$, then
$y_1 \geq 1$, $y_j \geq j$, and $x_i \geq i-1$, whence  
\begin{equation}\label{rk2}
  n \geq x_k + \sum^{k-1}_{i=1}(i-1) + \sum^l_{j=1}j -\frac{(k+l)(k+l-2)}{4}
  = x_k + \frac{(t-2)^2}{4}.
\end{equation}
If $x_1 \geq 1$ (including the case $k = 1$), then $x_i \geq i$, 
$y_j \geq j-1$, and so  
\begin{equation}\label{rk3}
  n \geq x_k + \sum^{k-1}_{i=1}i + \sum^l_{j=1}(j-1) -\frac{(k+l)(k+l-2)}{4}
  = x_k + \frac{t^2}{4}.
\end{equation}
In particular, (\ref{rk2}) and (\ref{rk3}) imply that $x_k \leq n$ and 
similarly $y_l \leq n$. Without loss we may also assume that $x_k \geq y_l$.

\smallskip
3) Under our assumptions on $(n,q)$, one can check that, if 
$1 \leq i \leq n$ and $\eps = \pm$, then $r|(q^i - \eps)$ only when 
$(i,\eps) = (n-1,+)$ and $s|(q^j-\eps)$ only when $(j,\eps) = (n-1,-)$. 
By 2), for any hook $(b,c)$ we have 
$1 \leq c-b \leq n$.  Now if $c-b \neq n-1$ for all hooks $(b,c)$, then 
(\ref{deg}) implies that $\chi$ has $r$-defect $0$ and so $\chi(x) = 0$,
a contradiction. So $S$ must admit a hook $(i,n-1+i)$ for some $i = 0,1$.
Similarly, $S$ possesses a cohook $(j,n-1+j)$ for some $j = 0,1$. In particular,
$x_k \geq n-1$.

\smallskip
4) Consider the case $x_k > n-1$, whence $x_k = n$ by 2). If $x_1 = 0$,
then (\ref{rk2}) implies that $t=2$, contradicting the condition $4|t$. 
So $x_1 > 0$, and so (\ref{rk3}) (and its proof) implies that $t = 0$, $k = l$,
$$X = \{1,2, \ldots ,k-1,n\}, ~~Y = \{0,1, \ldots ,k-1\}.$$ 
Now if $k = 1$, then $S = (\{n\},\{0\})$, yielding $\chi = 1_G$. If $k = n$, then 
$$X = \{1,2, \ldots ,n\},~~Y = \{0,1, \ldots ,n-1\},$$
yielding $\chi = \St$. In the remaining cases, $2 \leq k \leq n-1$, and so
$S$ cannot admit any hook of the form $(i,n-1+i)$.

\smallskip
5) Assume now that $x_k = n-1 \geq y_l$; in particular, $(0,n-1)$ is both a hook 
and a cohook for $S$. Suppose first that $x_1 = 0$. In this case, (\ref{rk2}) and 
the condition $4|t$ imply that $t = 0$ or $4$, 
$$X = \{0,1, \ldots ,k-2,n-1\},~~Y = \{1,2, \ldots ,l\}.$$ 
Also, since $0 \in X$ and $(0,n-1)$ is a hook, we have $n-1 \in Y$, implying 
$y_l = l = n-1$. Next, $k \leq n$ and $k-l = t \in \{0,4\}$, so $k = n-1$,
$$X = \{0,1, \ldots ,n-4,n-3,n-1\},~~Y = \{1,2, \ldots ,n-2,n-1\},$$
leading to the character $\beta$ which has the degree listed in the 
proposition, as one can see using (\ref{deg}). 

Suppose now that $x_1 \geq 1$. Since $4|t$, (\ref{rk3}) and its proof 
imply that $k = l$, and one of the following two cases occurs:

\smallskip
(a) $X = \{1,2, \ldots ,k-1,n-1\}$, $Y = \{0,1,\ldots ,k-2,k\}$, with 
$1 \leq k \leq n-1$, or 

\smallskip
(b) $X = \{1,2, \ldots ,k-2,k,n-1\}$, $Y = \{0,1,\ldots ,k-1\}$, with 
$2 \leq k \leq n-2$. 

\smallskip
In the case of (a), if $k = 1$ then $S = (\{n-1\},\{1\})$ and 
$\chi = \al$. On the other hand,
if $2 \leq k \leq n-1$, then $(0,n-1)$ can be a cohook for 
$S$ only when $k = n-1$, leading again to $\chi = \beta$. In the case of (b),
$(0,n-1)$ cannot be a cohook of $S$.

\smallskip
6) We have shown that $\chi \in \{1_G,\al,\beta,\St\}$. It remains to prove
that $\chi(x),\chi(y) \in \{1,-1\}$. The statement is obvious if $\chi = 1_G$ or
$\chi = \St$. Consider the case $\chi = \al$. Note that $\Irr(G)$ contains 
a unique irreducible character of degree $\al(1)$. 
%which is the unipotent character labeled by the symbol 
%$\begin{pmatrix}n-1\\1 \end{pmatrix}$. 
(Indeed, the claim is a consequence of \cite[Theorem 7.6]{TZ1} if $q \geq 4$, and it follows from \cite[Theorem 1.3]{N} if $q < 4$.) On the other hand,
it is well known (see e.g. \cite[Table 1]{ST}) that the rank $3$ permutation 
character $\rho$ of $G$ (acting on the singular $1$-spaces of the natural module 
$V = \FF_q^{2n}$) is the sum of $1_G$, an irreducible character of degree $\al(1)$,
and another one, say $\gam$, of degree $(q^{2n}-q^2)/(q^2-1)$. It follows that
$\rho = 1_G + \al + \gam$. Note that $\gam$ has $r$-defect $0$ and $s$-defect $0$.
Also, it easy to see that $\rho(x) = 2$ and $\rho(y) = 0$. Hence 
$\al(x) = \rho(x)-1 = 1$ and $\al(y) = \rho(y)-1 = -1$.

\smallskip
7) To prove the statement in the case $\chi = \beta$, we use the Alvis-Curtis 
duality functor $D_{\GC}$ which sends any irreducible character of $G$ to an 
irreducible character of $G$ up to a sign, cf. \cite[Corollary 8.15]{DM}. In 
the case of an $F$-stable torus $\TC$, 
%there is some $\eps_{\TC} = \pm 1$ such that 
$D_{\TC}(\lam) = \lam$ for all $\lam \in \Irr(\TC^F)$, see 
\cite[Definition 8.8]{DM}. Applying this and \cite[Corollary 8.16]{DM} to
$\TC_1 = C_{\GC}(x)$ (so that $T_1 = \TC_1^F$), we now see that there is some 
$\eps_{\GC} = \pm 1$ such that
$$\eps_{\GC}D_{\GC}(\al)(x) = \pm (D_{\TC_1} \circ \Res^G_{T_1})(\al)(x) 
  = \pm \al(x)  = \pm 1,$$
i.e. $D_{\GC}(\al)(x) = \pm 1$. Similarly, $D_{\GC}(\al)(y) = \pm 1$. 
In particular, by the results proved above, there is
some $\eps = \pm 1$ such that $\eps D_{\GC}(\al) \in \{1_G,\al,\beta,\St\}$. 

As shown in 6), $\alpha$ is a constituent of 
$\rho$. Hence it is also a constituent of the permutation character $1^G_B$, where $B$ is a Borel subgroup of $G$, and the same is true for $1_G$ and $\St$. For each 
irreducible constituent $\varphi$ of $1^G_B$, there is a polynomial 
$d_\varphi(X) \in \QQ[X]$ in variable $X$ (the so-called {\it generic degree}, cf.
\cite[\S13.5]{C}, which depends only on the Weyl group of $\GC$ but not on $q$) such that $\varphi(1) = d_\varphi(q)$. 
According to Theorem  (1.7) and Proposition (1.6) of \cite{Cur}, $D_\GC$ permutes 
the irreducible constituents of $1^G_B$. Moreover, there is an integer $N$ such that
\begin{equation}\label{d-power}
  d_{D_\GC(\varphi)}(X) = X^N d_\varphi(X^{-1}).
\end{equation}  
It is well known, see e.g. Corollary 8.14 and
Definition 9.1 of \cite{DM}, that $D_{\GC}$ interchanges $1_G$ and $\St$. Since
$\St(1) = q^{n(n-1)/2}$, (\ref{d-power}) applied to $\varphi = 1_G$ yields that 
$N = n(n-1)$. Applying (\ref{d-power}) to $\varphi = \al$, we now obtain that 
$D_\GC(\al)(1) = \beta(1) \neq \al(1)$. It follows that $D_\GC(\al) = \beta$.
Thus we have shown that $\beta(x), \beta(y) \in \{\pm 1\}$.
\end{proof}

Now we can complete the case $n=4$:

\begin{lem}\label{n=4}
{\sl If $n = 4$, then $x^G\cdot y^G = G \setminus Z(G)$.}
\end{lem}

\begin{proof}
It suffices to prove (\ref{count}) for every non-semisimple 
$g \in G \setminus Z(G)$. For such a $g$, $\St(g) = 0$. Furthermore,
inspecting the character values of $\al$ and $\beta$ as given in 
{\sf {Chevie}} \cite{Chevie}, we see that 
$$\left|\frac{\al(g)}{\al(1)}\right| + \left|\frac{\beta(g)}{\beta(1)}\right|
  \leq \frac{2q^3+q}{q(q^2+1)^2} + \frac{q^7}{q^7(q^2+1)^2} < 1/2,$$
and so we are done by Proposition \ref{values}.
\end{proof}

Next we estimate the character ratios $|\al(g)/\al(1)|$ for the character $\al$ 
described in Proposition \ref{values}.
  
\begin{lem}\label{bound1}
{\sl Assume $q$ is odd and $n \geq 6$. Then $|\al(g)/\al(1)| < 0.4$ for 
all $g \in G \setminus Z(G)$.}
\end{lem}

\begin{proof}
As mentioned in p. 6) of the proof of Proposition \ref{values}, $\al$ is 
the unique irreducible character of $G$ of degree $\al(1)$. Hence we may 
assume that $\al$ is the character $D^\circ_{1_S} = D_{1_S}-1$ of 
$\bG = \Omega^+_{2n}(q)$ constructed in \cite[\S5]{LBST1} using 
the dual pair $\bG * S$ inside $Sp_{4n}(q)$, with $S := Sp_2(q)$. In particular,
$$\al(g) = \frac{1}{|S|}\sum_{x \in S}\om(xh)-1$$
if $h \in \bG$ corresponds to $g$, and $\om$ is a {\it reducible Weil character}
of degree $q^{2n}$ of $Sp_{4n}(q)$. Denote
$$m(h) = \max_{\lam \in \FF_{q^2}}\dim \Ker(h-\lam \cdot 1_V),$$
where $V = \overline{\FF}_q^{2n}$. Since $g \notin Z(G)$, $m(h) \leq 2n-1$. 

First assume that $m(h) \leq 2n-4$. Arguing as in the proof of 
\cite[Proposition 5.11]{LBST1}, one sees that $|\al(g)| \leq q^{2n-4}+1$.
On the other hand, $\al(g) > q^{2n-3}$. Hence 
$$\left|\frac{\al(g)}{\al(1)}\right| < \frac{q^{2n-4}+1}{q^{2n-3}} < 0.4$$
as $n \geq 6$ and $q \geq 3$. 

Next suppose that $m(h) \geq 2n-3$. Then $h$ has an 
eigenvalue $\lam_0 = \pm 1$ such that 
$\dim \Ker(h-\lam_0 \cdot 1_V) = m(h)$, and 
$\dim \Ker(h-\lam \cdot 1_V) \leq 3$ for all 
$\lam \in \FF_{q^2} \setminus \{\lam_0\}$. Using \cite[Lemma 5.9]{LBST1} and
arguing as in the proof of \cite[Proposition 5.11]{LBST1}, we see 
that 
$$|\om(xh)| \leq \left\{ \begin{array}{ll}
  q^{2n-1},& x = \begin{pmatrix}\lam_0 & 0 \\0 & \lam_0\end{pmatrix},\\
  q^n, & x \mbox{ is $GL_2(q)$-conjugate to }
         \begin{pmatrix}\lam_0 & 1 \\0 & \lam_0\end{pmatrix},\\
  q^3, & \mbox{otherwise}
  \end{array} \right.$$  
for all $x \in S$. It follows that
$$|\al(g)| \leq 1+ \frac{q^{2n-1}+(q^2-1)q^n + (q(q^2-1)-q^2)q^3}{q(q^2-1)}
  < \frac{q^{2n-2} +q^{n+1}+q^5}{q^2-1}$$
and so
$$\left|\frac{\al(g)}{\al(1)}\right| < 
  \frac{q^{2n-2}+q^{n+1}+q^5}{(q^{n}-1)(q^{n-1}+q)} < 0.4$$
as well.
\end{proof}

\begin{lem}\label{bound2}
{\sl Assume $q$ is even and $n \geq 6$. Then $|\al(g)/\al(1)| < 0.4$ for 
all $g \in G \setminus \{1\}$.}
\end{lem}

\begin{proof}
Again consider the rank $3$ permutation character $\rho = 1_G +\al +\gam$ of 
$G = \Omega^+_{2n}(q)$. Our proof relies on the following key formula proved 
in \cite{GMT}:
\begin{equation}\label{rank3}
 \gam(g) = \frac{1}{2}\left( 
    \frac{1}{q-1}\sum^{q-2}_{i=0}q^{\dim\Ker(g-\delta^i \cdot 1_{\tV})} - 
    \frac{1}{q+1}\sum^{q}_{j=0}(-q)^{\dim\Ker(g-\xi^j \cdot 1_{\tV})}\right)-1
\end{equation}
for some $\delta \in \FF_q^{\times}$ of order $q-1$ and some 
$\xi \in \FF_{q^2}^{\times}$ of order $q+1$. Here, $V = \FF_q^{2n}$
is the natural module for $G$ and 
$\tV = V \otimes_{\FF_q}\overline{\FF}_q$. As before, we define
$$m(g) = \max_{\lam \in \FF_{q^2}}\dim \Ker(g-\lam \cdot 1_{\tV}).$$
Using (\ref{rank3}), it was shown in \cite{GMT} that 
$|\al(g)| \leq q^{m(g)}$. In particular, if $m(g) \leq 2n-4$, then
$$\left|\frac{\al(g)}{\al(1)}\right| \leq 
  \frac{q^{2n-4}(q^2-1)}{(q^n-1)(q^{n-1}+q)} < 0.4$$
since $n \geq 6$. 

\smallskip
Thus we may assume that $m(g) \geq 2n-3$; in particular, 
$\dim \Ker(g-1_V) = m(g)$. But $\dim \Ker(g-1_V)$ is even since 
$g \in \Omega^+_{2n}(q)$, cf. \cite[p. xii]{Atlas}. Also, $m(g) < 2n$ as
$g \neq 1$. It follows that $\dim \Ker(g-1_V) = 2n-2$. Now we can have 
the following two possibilities.

\smallskip
(a) The multiplicity of $1$ as an eigenvalue of $g$ (acting on $V$) is 
$2n-2$. In this case, $g = I_{2n-2} \oplus h$ for some semisimple
element $1 \neq h \in Sp_2(q)$. Thus $h$ is conjugate (over $\overline{\FF}_q$)
to $\diag(\lam,\lam^{-1})$ for some $\lam \in \overline{\FF}_q^{\times}$ 
with $1 \neq \lam^q \in \{\lam,\lam^{-1}\}$. 

(a1) Suppose that $\lam^q = \lam$. Then one can check that 
$$\rho(g) = 2+\frac{(q^{n-1}-1)(q^{n-2}+1)}{q-1}$$
(since $\Ker(g-1_V)$ is a non-degenerate subspace of dimension $2n-2$ of 
type $+$). On the other hand, (\ref{rank3}) yields 
$\gam(g) = (q^{2n-2}-1)/(q^2-1)$. It follows that
$$\al(g) = 1+\frac{(q^{n-1}-1)(q^{n-2}+q)}{q^2-1}.$$

(a2) Assume now that $\lam^q = \lam^{-1}$. Then 
one can check that 
$$\rho(g) = \frac{(q^{n-1}+1)(q^{n-2}-1)}{q-1}$$
(since $\Ker(g-1_V)$ is a non-degenerate subspace of dimension $2n-2$ of 
type $-$). On the other hand, (\ref{rank3}) again yields that 
$\gam(g) = (q^{2n-2}-1)/(q^2-1)$. It follows that
$$\al(g) = -1+\frac{(q^{n-1}+1)(q^{n-2}-q)}{q^2-1}.$$
In both of these subcases, $|\al(g)/\al(1)| < 0.4$ as $n \geq 6$.

\smallskip
(b) The multiplicity of $1$ as an eigenvalue of $g$ (acting on $V$) is 
$\geq 2n-1$. Since this multiplicity is even, it must equal $2n$, i.e. 
$g$ is unipotent. As $\dim \Ker(g-1_V) = 2n-2$, we see that 
$g = 2J_2 \oplus I_{2n-4}$, where $J_2$ denotes a Jordan block of size
$2$ with eigenvalue $1$, and furthermore $g$ acts trivially on 
a non-degenerate $(2n-4)$-dimensional subspace $U$ of $V$. 
By (\ref{rank3}) we have $\gam(g) = (q^{2n-2}-q^2)/(q^2-1)$. Let 
$Q$ denote the $G$-invariant quadratic form on $V$. We can now 
distinguish two subcases.

(b1) $U^\perp$ is decomposable as a sum of proper nonzero non-degenerate 
$g$-invariant subspaces. By \cite[Theorem 2.5]{FST}, there is a 
unique $G$-conjugacy class of elements with this property. So without
loss we may assume that $U$ has type $+$ and 
there is a symplectic basis $(e_1,f_1,e_2,f_2)$ of $U^\perp$ such that
$$g~:~e_1 \mapsto e_1,~~e_2 \mapsto e_2,~~f_1 \mapsto e_1+f_1,
  ~~f_2 \mapsto e_2+f_2,$$
and 
$$Q(e_1) = Q(e_2) = 1,~~Q(f_1) = Q(f_2) = 0.$$
Then $\Ker(g-1_V) = \langle e_1,e_2 \rangle_{\FF_q} \oplus U$. Observe that 
$\langle e_1,e_2 \rangle_{\FF_q}$ contains exactly $q$ singular vectors and
$q$ vectors $v$ with $Q(v) = 1$. Next, $U$ contains exactly 
$(q^{n-2}-1)(q^{n-3}+1)$ nonzero singular vectors and
$(q^{n-2}-1)q^{n-3}$ vectors $u$ with $Q(u) = 1$. It now follows by direct count
that the number of $g$-fixed singular $1$-spaces in $V$ is 
$$\rho(g) = 1+q \cdot (q^{n-2}-1)q^{n-3} + 
  q \cdot \frac{(q^{n-2}-1)(q^{n-3}+1)}{q-1} = 
  \frac{q^{2n-3}-1}{q-1}.$$    

(b2) $U^\perp$ is indecomposable as a sum of proper nonzero non-degenerate 
$g$-invariant subspaces. By \cite[Theorem 2.5]{FST}, there is a 
unique $G$-conjugacy class of elements with this property. So without
loss we may assume that $U$ has type $+$ and 
there is a symplectic basis $(e_1,f_1,e_2,f_2)$ of $U^\perp$ such that
$$g~:~e_1 \mapsto e_1,~~e_2 \mapsto e_1+e_2,~~f_1 \mapsto f_1+f_2,
  ~~f_2 \mapsto f_2$$
(so that $g|_{U^\perp}$ is a short-root element of $Sp(U^{\perp})$), and 
$$Q(e_1) = Q(e_2) = Q(f_1) = Q(f_2) = 0.$$ 
Then $\Ker(g-1_V) = \langle e_1,f_2 \rangle_{\FF_q} \oplus U$. Now 
$\langle e_1,f_2 \rangle_{\FF_q}$ is totally singular, and, as before,
$U$ contains exactly $(q^{n-2}-1)(q^{n-3}+1)$ nonzero singular vectors. 
It now follows by direct count that the number of $g$-fixed singular 
$1$-spaces in $V$ is 
$$\rho(g) = (q+1)+ q^2 \cdot \frac{(q^{n-2}-1)(q^{n-3}+1)}{q-1} = 
  \frac{q^{2n-3}+q^n-q^{n-1}-1}{q-1}.$$   
In both of these subcases,
$$|\al(g)| = |\rho(g)-1-\gam(g)| \leq \frac{q^{2n-3}+q^{n+1}-q^{n-1}-q}{q^2-1}.$$
It follows that $|\al(g)/\al(1)| < 0.4$ as well. 
\end{proof}

The main result of this subsection is the following 

\begin{thm}\label{D-type}
{\sl Let $G = Spin^+_{2n}(q)$ with $2|n \geq 4$, $(n,q) \neq (4,2)$, and let 
$x_1 \in T_1$ and $x_2 \in T_2$ be regular semisimple elements, 
where the tori $T_1$ and $T_2$ are described at the beginning of \S2.1. Then
$x_1^G\cdot x_2^G = G \setminus Z(G)$. In particular,
$x^G\cdot y^G = G \setminus Z(G)$, where $x$ and $y$ are regular semisimple of 
order $r$ and $s$.}
\end{thm} 

\begin{proof}
Note that it suffices to prove the statement for $x$ and $y$. Indeed,
the tori $T_1$ and $T_2$ are weakly orthogonal by Lemma \ref{tori}. Hence, by 
\cite[Proposition 2.2.2]{LST1}, all irreducible characters $\chi$ of $G$ 
that vanish neither on a regular semisimple element $x_1 \in T_1$ nor on 
a regular semisimple element $x_2 \in T_2$ must be unipotent. But then the 
results of \cite{DL} imply that $\chi(x_1)$ does not depend on
the particular choice of $x_1 \in T_1$ of given type, and similarly for 
$\chi(x_2)$; in particular, $\chi(x_1) = \chi(x)$ and $\chi(x_2) = \chi(y)$.
Hence, for any $g \in G$ we have
$$\sum_{\chi \in \Irr(G)}\frac{\chi(x_1)\chi(x_2)\overline{\chi}(g)}{\chi(1)} = 
  \sum_{\chi \in \Irr(G)}\frac{\chi(x)\chi(y)\overline{\chi}(g)}{\chi(1)}.$$
Consequently, $x_1^G\cdot x_2^G = x^G\cdot y^G$.

\smallskip
It remains to prove (\ref{count}) for every non-semisimple 
$g \in G \setminus Z(G)$. Applying Lemma \ref{n=4} we may assume 
that $n \geq 6$. Also, $\St(g) = 0$ for any such a $g$. Next,
$$\left|\frac{\al(g)}{\al(1)}\right| < 0.4$$ 
by Lemmas \ref{bound1} and \ref{bound2}. As in the
proof of \cite[Theorem 1.1.4]{LST1}, we have that 
$$|\beta(g)|^2 \leq |C_G(g)| \leq |G|/q^{2n-2} < q^{2n^2-3n+2}.$$
On the other hand, $\beta(1) >  q^{n^2-n-1}$. It follows that 
$$\left|\frac{\beta(g)}{\beta(1)}\right| < q^{2-n/2} \leq q^{-1} \leq 0.5.$$
Thus 
$$\left|\frac{\al(g)}{\al(1)}\right| + 
   \left|\frac{\beta(g)}{\beta(1)}\right| + 
   \left|\frac{\St(g)}{\St(1)}\right| \leq 0.9,$$
and so we are done by Proposition \ref{values}.    
\end{proof}

\subsection{Other Lie-type groups}
By Theorem \ref{D-type} (and the remark at the beginning of \S2.1),
Theorem \ref{main1} and Corollaries \ref{simple1}, \ref{simple2} hold for $G = Spin^{+}_{4n}(q)$. 
Now we will prove Theorem \ref{main1} and Corollaries \ref{simple1}, \ref{simple2} for the remaining
types. We will write $q=p^f$ as usual. In the cases where $s_1$ and $s_2$ can
be chosen to be equal, we write $s = s_1 = s_2$.

First we deal with a few special cases.

\begin{lem}\label{sl2}
{\sl Theorem \ref{main1} and Corollaries \ref{simple1}, \ref{simple2} hold for $G = SL_{2}(q)$
with $q \geq 4$.}
\end{lem}

\begin{proof}
Suppose first that $q \pm 1$ are not $2$-powers. Then we can choose odd 
prime divisors $r$ of $q-1$ and $s$ of $q+1$, and find   
regular semisimple elements $x \in G$ of order $r$ and 
$y \in G$ of order $s$. (In fact, we can choose
$r = \ppd(p,f)$ and $s = \ppd(p,2f)$ if $f \geq 2$ and $q \neq 8,64$.) 
Using the character table of $G$ (see e.g. \cite[\S38]{Do}), 
one can check that $x^G \cdot y^G = G \setminus Z(G)$.  

Suppose $q-1$ is a $2$-power. If $q = 9$, we can take $x,y$ to be non-conjugate 
elements of order $5$ in $G$. Otherwise $q$ is a Fermat prime. 
If $q=5$ or $17$, we can choose $|x| = 3$ and $|y|=4$ (but note that  
there is no desired pair $(x,y)$ of square-free orders). For $S = PSL_2(q)$
with $q = 5$ or $17$, we have $S = x^S \cdot x^S$ where $|x| = 3$. On the other
hand, if $q > 17$ is a Fermat prime, then it is not difficult to 
show that $q+1$ has a prime divisor 
$r \geq 5$. Choosing $x \in G$ of order $r$ and using the character table of
$G$, we can check that $x^G \cdot (x^2)^G = G \setminus Z(G)$.

Suppose now that $q+1$ is a $2$-power, i.e. $q =2^t-1 \geq 7$ is a Mersenne prime.      
Choosing $x \in G$ of order $r=t$, one can check that $x^G \cdot x^G = G \setminus Z(G)$.
\end{proof} 

\begin{lem}\label{sl3}
{\sl Theorem \ref{main1} and Corollaries \ref{simple1}, \ref{simple2} hold for $G = SL_{3}(q)$.}
\end{lem}

\begin{proof}
Let $r = \ppd(p,3f)$ if $q \neq 2,4$, $r = 7$ if $q = 4$, and $r = 3$ if $q=2$. 
Then we can find a regular semisimple element $x \in G$ of order $r$. By \cite{Gow},
$x^G \cdot x^G$ contains all non-central semisimple elements of $G$. Direct computation
using \cite{Chevie} shows that $x^G \cdot x^G$ also contains all other non-central 
classes of $G$. 
\end{proof}

\begin{lem}\label{su3}
{\sl Theorem \ref{main1} and Corollaries \ref{simple1}, \ref{simple2} hold for $G = SU_{3}(q)$
with $q > 2$.}
\end{lem}

\begin{proof}
Note that $r = \ppd(p,6f)$ exists since $q > 2$. Then we can find a regular semisimple
element $x \in G$ of order $r$. Again, using \cite{Chevie} one can check that 
$x^G \cdot x^G \supseteq G \setminus Z(G)$. 
\end{proof}

\begin{lem}\label{sp4}
{\sl Theorem \ref{main1} and Corollaries \ref{simple1}, \ref{simple2} hold for $G = Sp_{4}(q)$
where $q \geq 5$ is odd.}
\end{lem}

\begin{proof}
We will follow the proof of Lemma \ref{sl2} and use the notation and the character
table of $G$ as described in \cite{Sri}. First we let $r = \ppd(p,4f)$ and consider
$x \in G$ belonging to the class $B_1((q^2+1)/r)$ (so that $x$ is regular semisimple
of order $r$). If both $q \pm 1$ are not $2$-powers, then we can find odd primes
$s_1 = \ppd(p,2f)$ and $s_2 = \ppd(p,f)$ and consider
$y \in B_2((q^2-1)/s_1s_2)$ of order $s_1s_2$. If 
$q = 9$, take $y \in B_4(2,4)$ of order $5$.
If $q \geq 5$ is a Fermat prime, consider the element $y \in B_4((q+1)/6,(q+1)/3)$
of order $6$. If $q \geq 7$ is a Mersenne prime, take $y \in B_3((q-1)/6,(q-1)/3)$
of order $6$. In all cases, one can check that $y$ is also regular semisimple, and 
$x^G \cdot y^G \supseteq G \setminus Z(G)$.      
\end{proof}

\begin{lem}\label{sp}
{\sl Theorem \ref{main1} and Corollaries \ref{simple1}, \ref{simple2} hold if $G$ is one of the following 
groups: 
$$\left\{\begin{array}{l}
   SU_4(2),~SL_6(2),~SL_7(2),~Sp_6(2),~\Omega^-_8(2),\\
   G_2(4),~\tw2 F_4(2)',~Sp_4(3),~Sp_8(3),~Spin_9(3).
  \end{array} \right.$$}
\end{lem}

\begin{proof}
1) Using \cite{GAP}, we can find a regular semisimple elements $x \in G = SU_4(2)$ 
of order $5$ such that $x^G \cdot x^G \supseteq G \setminus \{1\}$. 
Similarly, if $G = SL_6(2)$, respectively $SL_7(2)$, $Sp_6(2)$, $\Omega^-_8(2)$, $G_2(4)$, 
$\tw2 F_4(2)'$, we can choose $x=y$ of order $31$, $127$, $7$, $7$, $7$, and $13$,
respectively. 

\smallskip
The case $G = Sp_4(3)$ is a genuine exception to the main claims in Theorem 
\ref{main1} and Corollary \ref{simple1}. Using \cite{GAP} one can check that

(i) there is no pair $(x,y) \in G \times G$ such that 
$x^G \cdot y^G \supseteq G \setminus Z(G)$ and $|x|$, $|y|$ are square-free;

(ii) even though $\bar{x}^S \cdot \bar{y}^S \supseteq S \setminus \{1\}$ with
$|\bar{x}| = |\bar{y}| = 5$ in 
$S = G/Z(G)$, the pair $(\bar{x},\bar{y})$ does not lift 
to any pair $(x,y) \in G \times G$ with $x^G \cdot y^G \supseteq G \setminus Z(G)$;

(iii) however, $x^G \cdot y^G = G \setminus Z(G)$ and 
$y^G \cdot y^G = G$ if $|x| = 5$ and $|y| = 8$.   
 
\smallskip
3) Let $G = Sp_8(3)$ and consider a regular semisimple element $x \in G$ of order $41$. 
Next, let $v \in Sp_4(3)$ be of order $10$ and let 
$$y := \diag(v,v^2) \in Sp_4(3) \times Sp_4(3) \hookrightarrow G.$$ 
Then $y$ is also regular semisimple, of order $10$ both in $G$ and in 
$S = G/Z(G) = PSp_8(3)$. Moreover, if $z$ denotes the central involution of 
$G$, then $y$ and $yz$ are conjugate in $G$. It follows that all faithful 
irreducible characters of $G$ vanish on $y$. The character table of $S$ (not of $G$!) is 
available in \cite{GAP}. One can now check that $y$ belongs to the class $10c$ in 
$S$, and there are precisely three irreducible characters of $S$ which are 
nonzero at both $x$ and $y$: $1_S$, $\St$, and $\al$ of degree $235,872$. Moreover,
$$\St(x) = -1,~~\St(y) = 1,~~\max_{g \in G \setminus Z(G)}|\St(g)| = 3^{10}$$ 
and
$$\al(x) = -1,~~\al(y) = 2,~~\max_{g \in G \setminus Z(G)}|\al(g)| = 29,484.$$
Thus for any $g \in G \setminus Z(G)$ we have 
$$\left|\sum_{\chi \in \Irr(G)}\frac{\chi(x)\chi(y)\overline{\chi}(g)}{\chi(1)}\right|
  \geq 1 - \frac{2 \cdot 29,484}{235,872} - \frac{3^{10}}{3^{16}} > 0.7,$$
i.e. $x^G \cdot y^G = G \setminus Z(G)$.  

\smallskip
4) Let $G = Spin_9(3)$ and consider a regular semisimple element $x \in G$ of order $41$.
We can also find $y \in G$ of order $39$ both in $G$ and in $S = G/Z(G) = \Omega_9(3)$.
Furthermore, there is $y' \in G$ which has order $26$ in $S$, and 
$y'$ is regular semisimple.
Clearly, if $\chi \in \Irr(G)$ and $\chi(x)\chi(y) \neq 0$ or 
$\chi(x)\chi(y') \neq 0$, then $\chi(1)$ is 
coprime to $13 \cdot 41$. The character table of $G$ is still unknown, but the degrees of 
irreducible characters of $G$ have been determined by F. L\"ubeck \cite{Lu}. Now we can 
check that there are precisely four irreducible characters of $G$ of degree coprime to
$13 \cdot 41$: $1_G$, $\al$ of degree $1,680$, $\beta$ of degree $11,022,480$, and 
$\St$. The character table of $S$ is 
available in \cite{GAP}, and $S$ also has irreducible characters of these four 
degrees. Thus the four aforementioned irreducible characters are actually 
trivial at $Z(G)$.
Again using \cite{GAP} one can check that
$$\al(y) = \St(y) = 0,~~\al(x)=\beta(x)=-1,~~\beta(y) = 1,~~
  \max_{g \in G \setminus Z(G)}|\beta(g)| = 408,240,$$
and 
$$|\al(y')| = |\beta(y')|=|\St(y')| = 1,~~
  \max_{g \in G \setminus Z(G)}|\al(g)| = 560,~~\max_{g \in G \setminus Z(G)}|\St(g)| = 3^{12}.$$
Thus for any $g \in G \setminus Z(G)$ we have 
$$\left|\sum_{\chi \in \Irr(G)}\frac{\chi(x)\chi(y)\overline{\chi}(g)}{\chi(1)}\right|
  \geq 1 - \frac{408,240}{11,022,480} > 0,$$
i.e. $x^G \cdot y^G = G \setminus Z(G)$. Similarly,
$$\left|\sum_{\chi \in \Irr(G)}\frac{\chi(x)\chi(y')\overline{\chi}(g)}{\chi(1)}\right|
  \geq 1 - \frac{560}{1,680}-\frac{408,240}{11,022,480}-\frac{3^{12}}{3^{16}} > 0.5$$
for all non-central $g \in G$,
whence $x^G \cdot (y')^G = G \setminus Z(G)$.
Furthermore, using \cite{GAP} one can check that $x^S \cdot x^S = S$. 
Thus we can use the pair $(x,y')$ for Theorem \ref{main1} and Corollary \ref{simple2},
and the pair $(x,y)$ for Corollary \ref{simple1}. 
\end{proof}

In what follows we will assume that $G$ is not isomorphic to any of the groups listed 
in Lemmas \ref{sl3}--\ref{sp}.

\subsubsection{Type $A_m$ with $n \geq 3$}
Let $G = SL_n(q)$ with $n \geq 4$, $(n,q) \neq (6,2)$, $(7,2)$. 
We aim to find $x$ and $y$ 
contained in tori $T_1$ of order $(q^n-1)/(q-1)$ and $T_2$ of order 
$q^{n-1}-1$. To this end, choose $r = \ppd(p,nf)$, 
and $s = \ppd(p,(n-1)f)$ if $(n,q) \neq (4,4)$ and $s = 7$ otherwise.
In all cases, it is easy to check that there exist regular semisimple 
elements $x \in T_1$ of order $r$ and $y \in T_2$ of order $s$.
In fact, any element of order $r$ in $G$ is regular semisimple
(and the same holds in all subsequent cases of our proof).
Now the tori $T_1$ and $T_2$ are weakly orthogonal (cf. 
\cite[Proposition 2.1]{MSW} or \cite[Proposition 2.3.1]{LST1}). Hence,
by \cite[Proposition 2.2.2]{LST1}, if $\chi \in \Irr(G)$ is nonzero at both $x$ 
and $y$ then $\chi$ is unipotent. This in turn implies by \cite{DL} that 
the value of $\chi$ at any regular semisimple element in $T_i$ does not depend
on the particular choice of the element. Hence we can apply 
\cite[Theorem 2.1]{MSW} to conclude that $x^G\cdot y^G \supseteq G \setminus Z(G)$.
(In subsequent cases we will frequently allude to this argument without 
mentioning it explicitly.) 

\subsubsection{Type $\tw2 A_m$ with $m \geq 3$}
Let $G = SU_n(q)$ with $n \geq 4$ and $(n,q) \neq (4,2)$.
First we consider the case $n \geq 5$ is odd. Then  
we can choose $r = \ppd(p,2nf)$ and find a regular 
semisimple element $x \in G$ of order $r$ that belongs to a maximal torus $T_1$ of 
order $(q^n+1)/(q+1)$. Next, if $n \equiv 1 (\mod 4)$ then we choose
$s = \ppd(p,(n-1)f)$. When $n \equiv 3 (\mod 4)$, we choose 
$s = \ppd(p,(n-1)f/2)$ if $(n,q) \neq (7,2^2)$, and 
$s = 7$ otherwise. One can show that there is a regular 
semisimple element $y \in G$ of order $s$ that belongs to a maximal torus $T_2$ of 
order $q^{n-1}-1$. By \cite[Theorem 2.2]{MSW} we have 
$x^G\cdot y^G \supseteq G \setminus Z(G)$.

Suppose now that $n \geq 4$ is even.  
Then we can find a regular semisimple element $x$ of order $r$ 
that belongs to a maximal torus $T_1$ of 
order $q^{n-1}+1$, where $r = \ppd(p,2(n-1)f)$. Next, 
if $n \equiv 0 (\mod 4)$ then we choose
$s = \ppd(p,nf)$. When $n \equiv 2 (\mod 4)$, we choose 
$s = \ppd(p,nf/2)$ if $(n,q) \neq (6,2^2)$, and 
$s = 7$ otherwise. One can show that there is a regular 
semisimple element $y \in G$ of order $s$ that belongs to a maximal torus $T_2$ of 
order $(q^n-1)/(q+1)$. Applying \cite[Theorem 2.2]{MSW} we see that 
$x^G\cdot y^G \supseteq G \setminus Z(G)$.

\subsubsection{Types $B_n$ and $C_n$ with $n \geq 2$}
Suppose that $G = Spin_{2n+1}(q)$ or $Sp_{2n}(q)$, with $n \geq 2$,
$(n,q) \neq (2,2)$, $(2,3)$, $(3,2)$, $(4,3)$.  
We aim to find $x$ and $y$ contained in tori $T_1$ of order $q^n+1$
and $T_2$ of order $q^n-1$. To this end, we take $r = \ppd(p,2nf)$. 
If $n$ is {\it odd}, then we choose $s = \ppd(p,nf)$ if 
$(n,q) \neq (3,4)$ and $s = 7$ if $(n,q) = (3,4)$. It is easy to check
that there exist regular semisimple elements $x \in T_1$ of order $r$ and
$y \in T_2$ of order $s$, and furthermore 
$x^G\cdot y^G \supseteq G \setminus Z(G)$ by \cite[Theorem 2.3]{MSW}. 

Assume that $2|n$ and $n \geq 4$. Then we choose $s_1 = \ppd(p,nf)$ if 
$(n,q) \neq (6,2)$ and $s_1 = 3$ if $(n,q) = (6,2)$.  Furthermore, if
$n \geq 6$, we take $s_2 = \ppd(p,nf/2)$ when $(n,q) \neq (12,2)$ and 
$s_2 = 7$ when $(n,q) = (12,2)$. If $n = 4$,  we choose 
$s_2 = \ppd(p,nf/2)$ whenever $q$ is not a Mersenne prime,
and $s_2 = 3$ if $q \geq 7$ is a Mersenne prime.
In all cases, one can check that there exist regular 
semisimple elements $x \in T_1$ of order $r$ and $y \in T_2$
of order $s_1s_2$.
(For instance, we can choose $y$ of order $91$ if $(n,q) = (12,2)$.
If $n = 4$ and $q=2^a-1 \geq 7$ is a Mersenne prime, then note that 
$Sp_8(q)$, respectively $\Omega_9(q)$, contains a cyclic subgroup of 
order $q^4-1$, respectively $(q^4-1)/2$. It follows in this case 
that $G$ contains a semisimple element $y$ of order $s_1s_2$, 
and it is easy to check that $y$ is regular.) 
Now, $x^G\cdot y^G \supseteq G \setminus Z(G)$ by
\cite[Theorem 2.3]{MSW}.    

Finally, assume that $n=2$ and $q \geq 4$. Since $Spin_5(q) \cong Sp_4(q)$ and
by Lemma \ref{sp4}, we may assume that $G = Sp_4(q)$ and $q=2^f$.  
Choose $s_1 = \ppd(2,2f)$ if $f \neq 3$ and $s_1 = 3$ if $f = 3$, and 
$s_2 = \ppd(2,f)$ if $f \neq 6$ and $s_2 = 3$ if $f = 6$. One 
readily checks that there exist regular semisimple elements $x \in T_1$
of order $r$ and $y \in T_2$ of order $s_1s_2$, and we are done as before.
(Note that the non-simple group $Sp_4(2)$ is excluded in the above analysis;
for $Sp_4(2)' \cong \AAA_6$ we can choose $r=s=5$.)

\subsubsection{Types $D_n$ and $\tw2 D_n$}
Note that the case of $D_n$ with $2|n$ is already completed by Theorem \ref{D-type}.
Assume now that $G = Spin^+_{2n}(q)$, where $n \geq 5$ is {\it odd}. Then we can 
choose $r = \ppd(p,nf)$ and find a regular 
semisimple element $x$ of order $r$ that belongs to a maximal torus $T_1$ of 
order $q^n-1$, see e.g. \cite[Lemma 2.4]{MT}. Similarly, we can find a regular 
semisimple element $y$ of order $s$ that belongs to a maximal torus $T_2$ of 
order $(q^{n-1}+1)(q+1)$, for some $s = \ppd(p,2(n-1)f)$.  By 
\cite[Theorem 2.6]{MSW} we have 
$x^G\cdot y^G \supseteq G \setminus Z(G)$.

Suppose now that $G = Spin^-_{2n}(q)$ with $n \geq 4$ and $(n,q) \neq (4,2)$. 
Then we can choose $r = \ppd(p,2nf)$ and find a regular 
semisimple element $x$ of order $r$ that belongs to a maximal torus $T_1$ of 
order $q^n+1$, see e.g. \cite[Lemma 2.4]{MT}. Similarly, we can find a regular 
semisimple element $y$ of order $s$ that belongs to a maximal torus $T_2$ of 
order $(q^{n-1}+1)(q-1)$, where $s = \ppd(p,2(n-1)f)$. 
Applying  \cite[Theorem 2.5]{MSW}, we conclude that 
$x^G\cdot y^G \supseteq G \setminus Z(G)$.

\subsubsection{Exceptional groups}
In the cases where $G = \tw2 B_2(q)$ with $q \geq 8$, respectively 
$G = \tw2 G_2(q)$ with $q \geq 27$, by \cite[Theorem 7.1]{GM} we can take   
$r = s = \ppd(2,4f)$, respectively $r = s = \ppd(3,6f)$. Similarly, in the cases 
where $G = G_2(q)$ with $q \neq 2$, respectively 
$G = \tw3 D_4(q)$, by \cite[Theorem 7.2]{GM} we can take   
$r = s = \ppd(p,3f)$ if $q \neq 4$ and 
$r=s=7$ if $q = 4$, respectively $r = s = \ppd(p,12f)$ (here, the existence
of regular semisimple elements of order $r$ follows from \cite[Lemma 2.3]{MT}). 
If $G = F_4(q)$, then $G$ contains regular semisimple elements 
$x \in G$ of order $r = \ppd(p,12f)$ and $y \in S$ of order 
$s = \ppd(p,8f)$ by \cite[Lemma 2.3]{MT}, and 
$x^G \cdot y^G = G \setminus \{1\}$ by \cite[Theorem 7.6]{GM}. Similarly,
if $G = E_8(q)$, then $G$ contains regular semisimple elements 
$x \in G$ of order $r = \ppd(p,24f)$ and $y \in G$ of order 
$s = \ppd(p,20f)$, and 
$x^G \cdot y^G = G \setminus \{1\}$ by \cite[Theorem 7.6]{GM}.

\smallskip
Suppose that $G = E_7(q)_{sc}$. By \cite[Lemma 2.3]{MT}, $G$ contains a 
regular semisimple element $x \in G$ of order $r = \ppd(p,18f)$ (and 
with centralizer of order $(q+1)(q^6-q^3+1)$). Furthermore, it is 
shown in the proof of \cite[Theorem 4.2]{HSTZ} that there is a regular 
semisimple $y \in G$ of order  $s = \ppd(p,7f)$. Now we can apply 
\cite[Theorem 7.6]{GM} to conclude that $x^G \cdot y^G = G \setminus Z(G)$. 

\smallskip
Next let $G = \GC^F = E^\eps_6(q)_{sc}$, with $\eps = +$ for 
$E_6(q)_{sc}$ and $\eps = -$ for $\tw2 E_6(q)_{sc}$. By \cite[Lemma 2.3]{MT}, 
$G$ contains a regular semisimple element $x \in G$ of order $r$, with $r = \ppd(p,9f)$ 
if $\eps = +$ and $r = \ppd(p,18f)$ if $\eps = -$ (and with centralizer of
order $q^6+\eps q^3+1$). Next we choose $s = \ppd(p,8f)$ and let
$y \in G$ be of order $s$. Applying \cite[Lemma 2.2]{MT}, we see that 
$C_\GC(y)$ is connected and $s$ divides $|(Z(C_\GC(y))^\circ)^F|$. The order of the latter
(for all $y$) is listed in \cite{Der}. Using this, one can easily check that 
$C_\GC(y)$ is a torus, i.e. $y$ is regular, and 
$|C_G(y)| = (q^4+1)(q^2-1)$. It then again follows by \cite[Theorem 7.6]{GM} 
that $x^G \cdot y^G = G \setminus Z(G)$. 

\smallskip
Finally, the case of $\tw2 F_4(q)$ with $q > 2$ follows from the following
statement:

\begin{lem}\label{2f4}
{\sl Let $G = \tw2 F_4(q)$ with $q = 2^f > 2$. Then $G$ admits regular 
semisimple elements $x$ of order $r = \ppd(2,12f)$ and $y$ of order $s = \ppd(2,6f)$,
such that $x^G \cdot y^G = G \setminus \{1\}$.}
\end{lem}

\begin{proof}
The existence of regular semisimple elements $x \in G$ of order $r$ 
and $y \in G$ of order $s$ is proved in \cite[Lemma 2.3]{MT}. In particular,
$|C_G(x)| = (q^2+q+1) + \eps\sqrt{2q}(q+1)$ for some $\eps = \pm$;
moreover, in the notation of \cite{M1}, $x$ is of type $t_{17}$ if $\eps = +$
and of type $t_{16}$ if $\eps = -$, whereas $y$ is of type $t_{15}$. 
Suppose now that $\chi \in \Irr(G)$ is nonzero at both $x$ and $y$, and
$\chi$ belongs to the Lusztig series $\EC(G,(t))$ labeled by the semisimple
element $t \in G^* \cong G$. Since $\chi(x) \neq 0$, $\chi$ cannot have 
$r$-defect zero, and so $|C_G(t)|$ is divisible by $r$. Similarly,
$|C_G(t)|$ is divisible by $s$. These condition imply that $t = 1$, i.e.
$\chi$ is unipotent. The values of unipotent characters of $G$ are determined
in \cite{M1}. An inspection of these values reveals that there are precisely
four possibilities for $\chi$: $1_G$, $\St$, and two more characters
of degree $q^2(q^4-1)^2/3$, labeled
by $\chi_{19}$ and $\chi_{20}$ in \cite{M1}. Moreover, 
$$\chi_{19}(x) = \chi_{20}(x) = 1,~~\chi_{19}(y) = \chi_{20}(y) = -1.$$ 

Now let $g \in G$ be any non-trivial element. As mentioned above, 
$g \in x^G \cdot y^G$ if $g$ is semisimple. If $g$ is not semisimple, then 
using \cite{M1} we see that
$$\left|\sum_{1_G \neq \chi \in \Irr(G)}
  \frac{\chi(x)\chi(y)\overline{\chi}(g)}{\chi(1)}\right| 
  = \left|\frac{\chi_{19}(g)+\chi_{20}(g)}{\chi_{19}(1)}\right|
  \leq \frac{2q^2(q^4-1)/3}{q^2(q^4-1)^2/3} = \frac{2}{q^4-1} < 0.1,$$
whence (\ref{count}) holds, and so we are done.  
\end{proof}

We have completed the proof of Theorem \ref{main1} and Corollary \ref{simple1}.

\subsection{Proof of Corollary \ref{simple2}}
In view of previous results, it remains to prove Corollary \ref{simple2} 
for alternating and sporadic simple groups. For these groups, the statement follows from

\begin{lem}\label{alt} 
{\sl Let $S$ be an alternating or sporadic finite simple group.
Then there is an element $x \in S$ of prime order $r$ such that 
$$x^S \cdot (x^{-1})^S = x^S \cdot x^S \cdot x^S = S.$$ 
Moreover, if $S$ is sporadic, then $r$ can
be chosen to be the largest prime divisor of $|S|$.}
\end{lem}

\begin{proof}
Note that if $x^S \cdot (x^{-1})^S = S$, and $x$ is real or 
$x^{-1} \in x^S \cdot x^S$, then 
$$S = x^S \cdot (x^{-1})^S \subseteq x^S \cdot x^S \cdot x^S.$$
Assume that $S = \AAA_n$ with $n \geq 5$.  
By the main result of \cite{B}, every $g \in S$ is 
a product of two $r$-cycles if $r \geq \lfloor 3n/4 \rfloor$, Moreover, if $r \leq n-2$ then
the $r$-cycles form a unique $\AAA_n$-class and they are all real. 
Hence we are done if the interval
$[\lfloor 3n/4 \rfloor,n-2]$ contains a prime. The latter claim 
holds for $n \geq 33$, since then $(5(n-2)/6,n-2)$ contains a prime. 
It also holds for $n \geq 5$ but $n \neq 6,8,11,12$ by 
direct inspection. In the cases $n = 6,8,11,12$, a direct computation 
using \cite{GAP} shows that 
\begin{equation}\label{choice}
  x^S \cdot (x^{-1})^S = S,~~~x^{-1} \in x^S \cdot x^S
\end{equation}
if we can choose $x$ of order $5$, $7$, $11$, and $11$, respectively. 
 
If $S$ is a sporadic group and $x \in S$ is an element of largest 
prime order, then (\ref{choice}) can be verified directly using \cite{GAP}.
\end{proof}     

\section{The Waring problem for powers}
\subsection{Proof of Theorem \ref{main2}}
Note that the statement is obvious if $k$ and $l$ are coprime. So we 
will assume that $\gcd(k,l) > 1$. Now if $m \leq 5$, then the latter 
condition implies that $m$ is also equal to ${\mathrm {lcm}}(k,l)$ and
that $m$ is a prime power. In this case, the statement follows from 
\cite[Corollary 1.5]{GM}, which says that every element in any finite 
non-abelian simple group is a product of two $m^{\mathrm {th}}$-powers,
provided that $m$ is a prime power. 

From now we will assume that $m \geq 6$. In particular, $m^{8m^2} > 10^{223}$,
and so we can ignore all the sporadic simple groups. Suppose that 
$S \cong \AAA_n$; in particular, $n > \max(4m,200)$. Under this assumption,
we can find a prime $p$ such that $5n/6 < p < n$ (see e.g. \cite{R});
in particular, $p > m$. By the main result of \cite{B}, every 
$g \in \AAA_n$ is a product of two $p$-cycles, whence it is a product of 
a $k^{\mathrm {th}}$-power and an $l^{\mathrm {th}}$-power.

Thus we may now assume that $S$ is a simple group of Lie type of order
$>10^{223}$ and   
view $S = G/Z(G)$, where $G = \GC(\FF_q) = \GCF$ as in Theorem \ref{main1}
and $q = p^f$. It suffices 
to show that every element $g \in G \setminus Z(G)$ is a product of two 
elements of orders coprime to both $k$ and $l$ in $G$. If the characteristic
$p$ of $G$ is larger than $m$, then the statement follows from 
\cite[Corollary, p. 3661]{EG}. So we may assume that $p \leq m$. Let 
$d$ denote the rank of the algebraic group $\GC$. By  
Theorem \ref{main1} and its proof, there exist primes $r,s_1,s_2$ such that 
$g = xy$ for some $r$-element $x \in G$ and $\{s_1,s_2\}$-element $y \in G$; 
moreover, $t > df/2$ if $G$ is classical and $t > df$ if $G$ is exceptional
for $t:=\min(r,s_1,s_2)$. Certainly, we are done if $t > m$. Suppose
that $t \leq m$. If $G$ is classical, then $df \leq 2t-1 \leq 2m-1$, and 
$$|S| < q^{d(2d+1)} \leq p^{df(2df+1)} < m^{(2m-1)(4m-1)} < m^{8m^2}.$$
If $G$ is exceptional, then $df < t \leq m$, and 
$$|S| < q^{31d} \leq p^{31df} < m^{31m} < m^{8m^2}$$
(since $m \geq 6$), completing the proof of Theorem \ref{main2}.
    
\smallskip    
The above proof of Theorem \ref{main2} also yields the following statement:

\begin{cor}\label{large}
{\sl {\rm (i)} Let $S$ be a finite non-abelian simple group. 
Then there exist primes $r,s_1,s_2$ such that 
every non-trivial element $g \in S$ is a product of an $r$-element $x \in S$ 
and an $\{s_1,s_2\}$-element $y \in S$. Moreover, the primes
$r$, $s_1$, and $s_2$ can be chosen to be arbitrarily large if $|S|$ is 
large enough.

\smallskip
{\rm (ii)} Let $\GC$ be a simple simply connected algebraic group in positive 
characteristic and let $F~:~\GC \to \GC$ be a generalized Frobenius endomorphism such that 
$G := \GC^F$ is quasisimple. Then there exist primes $r,s_1,s_2$ such that 
every non-central element $g \in G$ is a product of an $r$-element $x \in G$ 
and an $\{s_1,s_2\}$-element $y \in G$. Moreover, the primes
$r$, $s_1$, and $s_2$ can be chosen to be arbitrarily large if $|G|$ is 
large enough.}
\end{cor} 

\subsection{Further results on the width}
Recall that the main result of \cite{LST1} establishes {\it width $2$} for 
arbitrary non-trivial word maps on {\it sufficiently large} finite simple groups
$S$. What happens if one removes the condition on the order 
of $S$?

It has been shown in 
\cite{KN} that the width can grow unbounded even when $w(S) \neq \{1\}$;
namely, given any $N$, there is a word $w$ in the free group on two generators
and a finite simple group $S$ such that $w(S) \neq \{1\}$ but yet $w(S)^N \neq S$. 

The next two examples show that the same is true for powers.  

\begin{exa}\label{unb1}
{\em Let $p$ be any prime, and let $S = PSL_{p^a+1}(p^b)$ for some 
integers $a, b \geq 1$. Set $w(x) = x^k$ with $k := \exp(S)/p$. 
We claim that {\sl $w(S)$ consists 
of the identity and all transvections in $S$, and, consequently,
$w(S)^{p^a} \neq S$}. Indeed, note that the $p$-part $k_p$ of $k$ is $p^a$. 
Considering the Jordan decomposition $g = su$ for any $g \in S$, we see that 
$g^k = (su)^k = u^k$
is non-trivial precisely when $u$ is a Jordan block of size $p^a+1$, in which 
case $g^k=u^k$ is a transvection. Now if $h = g_1g_2 \ldots g_{p^a}$ and 
$g_i \in w(S)$, then, after lifting $g_i$ to a $p$-element 
$\hat{g}_i \in G = SL_{p^a+1}(p^b) = SL(V)$, we have that 
$\codim\ C_V(\hat{g}_i) \leq 1$. It follows that 
$\codim\ C_V(\hat{h}) \leq p^a$, i.e. $C_V(\hat{h}) \neq 0$ for  
$\hat{h} := \prod^{p^a}_{i=1}\hat{g}_i$. Hence $w(S)^{p^a} \neq S$.}
\end{exa} 

Of course similar examples hold for other classical groups (in characteristic 
$p$). We offer an example in cross characteristic as well:

\begin{exa}\label{unb2}
{\em Let $p$ be any prime, $a \geq 1$, and let $S = PSL_{p^a+1}(q)$ for some 
prime power $q$ such that $p|(q-1)$. For simplicity, assume in addition that
$p > 2$ and $(q-1)_p = p$. Again set $w(x) = x^k$ with $k := \exp(S)/p$. We claim 
that {\sl $w(S) \neq \{1\}$ and consists of the identity and some scalar 
multiples of pseudoreflections in $S$; furthermore,
$w(S)^{p^a} \neq S$}. To see this, we work in $G = SL_{p^a+1}(q) = SL(V)$ and 
again note that the $p$-part of $k$ is $p^a$. For any $g \in G$, we see that 
$g^k$ is non-trivial precisely when the $p$-part of $g$ is conjugate (over
$\overline{\FF}_q$) to
$$\diag(\lam,\lam^q, \ldots ,\lam^{q^{p^a-1}},\lam^{\frac{q^{p^a-1}}{1-q}}),$$
where $\lam \in \FF_{q^{p^a}}^{\times}$ has order $p^{a+1}$. Hence, $g^k$ is either  
$1$ or a pseudoreflection up to scalar. Now if $h = g_1g_2 \ldots g_{p^a}$ and 
$g_i \in w(G)$, then we have that 
$\codim\ \Ker(g_i-\lam_i \cdot 1_V) \leq 1$ for some $\lam_i \in \FF_q^{\times}$. 
It follows that 
$\codim\ \Ker(h-\mu \cdot 1_V) \leq p^a$, i.e. $\Ker(h - \mu\cdot 1_V) \neq 0$ 
for  $\mu := \prod^{p^a}_{i=1}\lam_i \in \FF_q^{\times}$. Hence $w(S)^{p^a} \neq S$.}
\end{exa}

Kassabov and Nikolov  \cite{KN}  gave more complicated examples with words that
were not powers (including an example for alternating groups).  

We next show that there are no such examples for alternating groups using powers. 

%%% remark likely X^3= S below, maybe X^2
\begin{lem}  \label{lem:altinv}  
{\sl Let $S = \AAA_n$ with $n \ge 5$, $\ell > 1$
an integer, and let $X$ be the subset of $S$ consisting of all
elements whose non-trivial orbits all have size $\ell$.
\begin{enumerate}
\item   If $\ell = 2$,  $X^3 = S$.
\item  If $\ell$ is odd, $X^4 = S$.
\end{enumerate}}
\end{lem}

\begin{proof}  
The first statement follows trivially from the elementary fact that
every element of  the symmetric group  is a product of two involutions.  

Now assume that $\ell$ is odd.  We will show that $X^2$ contains either an
$n$-cycle or $(n-1)$-cycle (depending upon whether $n$ is odd or even).
We may assume that $n$ is odd (replacing $n$ by $n-1$ if necessary).
Write $n= k \ell + r$ where $0 \le r < \ell$.  

We actually prove a slightly stronger statement by induction on $n$ (assuming
$n$ is odd).   An $n$-cycle can be written as a product of two elements of
$X$ one of which fixes a point (and so any specified point).
If $n=5$, this is clear and more generally if $n = \ell$, this is clear.

Let $x = (1,2, \ldots,n)$ and let $y = (n,n -1,\ldots ,n - \ell + 1)$.   Note
that $x y = (1,2, \ldots  ,n - \ell + 1 )$.  By induction,  $xy = uv $ is a product of
two elements of $X \cap H$ where $H$ is the subgroup of $S$ fixing
$\{n - \ell +2, \ldots, n\}$ and moreover, we may assume that 
$v$ fixes $n - \ell+1$.  Thus, $x = u (vy^{-1})  \in X^2$ and $u$ fixes a point
(indeed at least $\ell-1$ points). Applying this argument to get such an expression
$x^{-1} = u_1v_1$, we see that $x = v_1^{-1}u_1^{-1}$  with $u_1,v_1 \in X$ and 
$u_1^{-1}$ fixing a point.  

It follows by \cite{B} that $X^4 = S$.
\end{proof}

We can now show there is a small universal bound for products of powers
covering in alternating groups (as long as not every power is trivial). 

\begin{thm}  \label{thm:alt}  
{\sl Let  $k$ be a positive integer and let $S=\AAA_n, n \ge 5$.
Assume that $k$ is not a multiple of the exponent $e$ of $S$.
Then every element of $S$ is a product of $8$ $k^{\mathrm {th}}$  powers.}
\end{thm}

\begin{proof} 
Let $p$ be a prime dividing $e/\gcd(e,k)$ and  let $p^{a+1}$
be the largest power of $p$ dividing $e$.    Write $n = sp^{a+1} + r$
with $0 \leq r < p^{a+1}$.   Then any element which is a product of
$sp^a$ disjoint $p$-cycles is a $k^{\mathrm {th}}$  power.  Let $Y$ be the
set of such elements.  It is straightforward to see that $Y^2$ contains
the set of all elements of $S$ in which all non-trivial orbits have size $p$.
It follows by the previous result that  $Y^8 = S$  (if $p=2$, $Y^6=S$). 
\end{proof} 

Note that if $n=2^{a+1} - 1, a \ge 2$ and $k=e/2$ where $e$ is the exponent
of $\AAA_n$, then non-trivial  $k^{\mathrm {th}}$  powers are just the involutions
moving exactly $2^a$ points.  One sees that an $n$-cycle is not the product
of $3$ $k^{\mathrm {th}}$  powers. 

\smallskip
We can show that there is a universal bound for the finite simple groups of 
Lie type as well under a slightly stronger hypothesis.
We sketch the proof.  The constant in the next   results is most likely
at most $5$.

We point out two easy observations that we use below.  
In these two statements, by a finite quasisimple group of Lie type
we mean any quotient of the group $\GC^F$ in Theorem \ref{main1} by a central 
subgroup.

\begin{cor} \label{cor:regss}  
{\sl Let $G$ be a finite quasisimple group of Lie type and let $G_{\mathrm {rss}}$ 
be the set of regular semisimple elements in $G$.  
Then $G=(G_{\mathrm {rss}})^2$.}
\end{cor} 

\begin{proof}
This is a trivial consequence of \cite{GM} as well as Theorem \ref{main1}, which 
show that any non-central element of $G$ is contained in $(G_{\mathrm {rss}})^2$.
On the other hand, if $z \in Z(G)$ and $s \in G_{\mathrm {rss}}$, then 
$s^{-1}z \in G_{\mathrm {rss}}$ and so $z = s \cdot s^{-1}z \in (G_{\mathrm {rss}})^2$.
\end{proof}

\begin{cor} \label{cor:regss2}  
{\sl Let $G$ be a finite quasisimple group of Lie type and
let $C$ be a conjugacy class of regular semisimple elements in $G$. Then $C^4=G$.}
\end{cor}

\begin{proof} 
As we have already noted, it follows by \cite{Gow} that 
$C^2$ contains all non-central semisimple elements. In particular,
in the notation of Theorem \ref{main1} we have that 
$$C^4 = C^2 \cdot C^2 \supseteq x^G \cdot y^G \supseteq G \setminus Z(G).$$
Suppose now that $z \in Z(G)$. If $s \in G$ is regular semisimple then so is
$s^{-1}z$, whence $s, s^{-1}z \in C^2$ and $z = s \cdot s^{-1}z \in C^4$ as well.
\end{proof}

\begin{thm}  \label{thm:chev}    
{\sl Let $S$ be a finite simple group and let $p$ be a prime dividing $|S|$. If 
$X$ denotes the set of $p$-elements of $S$, then $X^{70}=S$.}
\end{thm}  

\begin{proof}  
If $S$ is sporadic, this is easily seen from the character tables.
If $S$ is an alternating group, it follows that $X^4=S$ by Lemma \ref{lem:altinv}.
If $S$ is a finite group of Lie type of rank at most $8$, this follows by \cite{LL}.  

So it suffices to prove the result for classical groups (of sufficiently large rank).
We give the proof for the case  $S=PSL_d(q)$ (with a better constant) and leave
the other cases to the reader.

If $p$ divides $q$, then $X^2 = S$ by \cite[Corollary, p. 3661]{EG}.  
So assume that $p$ does not
divide $q$. Let $P$ be a Sylow 
$p$-subgroup of $G$.   If $d \le 3$, then $P$ contains a regular semisimple
element of $S$ and so $X^2$ contains all semisimple elements, whence
$X^4 = S$.   So assume that $d \ge 4$.   If $p \le d + 1$, then 
by the proof of Lemma \ref{lem:altinv},  $X^2$ will contain  either a $d +1$
cycle or $d$-cycle of $H:=\AAA_{d+1} < S$.   If $d + 1$ is odd 
and $p$ does not divide $d+1$, then an $(n+1)$-cycle in $H$ is  a regular semisimple
element of $S$, whence $X^4=S$.  If $d + 1$ is odd and $p$ divides $d+1$,
then similarly, we see that $X^2$ contains a $(d-1)$-cycle which is semisimple
and has all eigenvalues of multiplicity $1$ aside from $1$ which has multiplicity 
$2$.  It follows that $X^4$ contains a regular semisimple element and so
$X^{16}=S$.   The same argument shows that $X^{16}=S$ for $d$ even as well. 

So assume that $p > d+1$. Let $V$ be the natural module for $SL_d(q)$
and lift $P$.  It is a straightforward exercise to show that $P$ contains an element
with distinct eigenvalues on $W:=[P,V]$ and that $\dim W > (1/2) \dim V$.
Thus, $X^2$ contains all semisimple elements of $\SL(W)$ and so
$X^3$ contains a regular semisimple element.  Thus, $X^{12}=S$.
\end{proof} 

A restatement of the previous result in terms of powers is:

\begin{cor} \label{cor:powers}  
{\sl Let $S$ be a finite simple group, and let $d$ be a positive integer such that
some prime $p$ divides $|S|$ but not $d$.  Then every element
of $S$ is a product of at most $70$ $d^{\mathrm {th}}$ powers.}
\end{cor} 

In fact, the same proof gives:

\begin{cor} \label{cor:bestex}  
{\sl Let $S$ be a finite simple group of Lie type.
Let $p$ be a prime which is not the characteristic of $S$ and does not
divide the order of a quasi-split torus.  If $d$ is a positive integer and
$p$ divides $\exp(S)/d$, then 
every element of $S$ is a product of at most $70$ $d^{\mathrm {th}}$ powers.}
\end{cor}

\begin{proof}  We assume that $p$ is not the characteristic.   If the rank of
$S$ is  at most $8$, the result follows by \cite{LL}.  So we may assume
that $S$ is classical.   We give the proof for $PSL_d(q)$.   First suppose
that $p \le d + 1$.  Since $p$ does not divide $q-1$, it follows that 
$x^d$ does not vanish on $\AAA_{d+1} < S$ and we argue as above.

So assume that $p > d+1$.    Then a Sylow $p$-subgroup  $P$ of $S$
is abelian.  Let $V$ denote the natural module for $S$.  So
$[P,V] =W_1 \oplus \ldots \oplus W_m$ where $P$ acts irreducibly on 
each $W_i$.  Since $p > \dim W_i$, it follows that any element of
$P$ that acts non-trivially on $W_i$ also acts irreducibly on $W_i$.
By hypotheses,  every element  of $H:=\Omega_1(P)$ is a $d^{\mathrm {th}}$ power.
Now choose $1 \ne x_i \in H, 1 \le i \le m$ so that the $x_i$ have distinct
eigenvalues (over the algebraic closure).  This is possible since all
$p^{\mathrm {th}}$-roots of $1$ occur as eigenvalues and $m \dim W_1 < p$.
Letting $X$ be the image of the word $x^d$, it follows that $X$ contains all
semisimple elements acting on $[P,V]$.  Since $\dim [P,V] > (1/2) \dim V$,
it follows that $X^3$ contains a regular semisimple element of $S$
and so $X^{12}=S$.
\end{proof} 

{\bf Proof of Corollary \ref{power}.} 
%As in the proof of Theorem \ref{main1}, we may assume that $m \geq 6$.
If $|S| \geq m^{8m^2}$ then we are done by
Theorem \ref{main2}. By Theorem \ref{thm:alt} we are also done if $S \cong \AAA_n$.
The statement is obvious if $S$ is abelian.
So we may assume that $|S| < m^{8m^2}$ and $S \not\cong \AAA_n$
(and $S$ is non-abelian). By the assumption,
there is some $1 \neq x \in S$ such that $x$ (and so $x^{-1}$ as well) is a 
$m^{\mathrm {th}}$ power in $S$. Suppose first that $S$ is a sporadic simple group.
As shown in \cite{Z}, the {\it covering number} $\cn(S)$ is at most $6$, and so each 
element $g \in S$ is a product of at most $6$ conjugates of $x$. Next assume
that $S$ is a simple group of Lie type, of untwisted Lie rank $r$. Then
$$2^{r^2} < |S| < m^{8m^2}$$
and so $r < m\sqrt{8\log_2 m}$. Now, according to the main
result of \cite{LL}, every $g \in G$ is a product of at most 
$$40r + 56 < 80m\sqrt{2\log_2 m} + 56 = f(m)$$ 
conjugates of  $x$ or $x^{-1}$, and so we are done. 

\section{Proof of Theorem \ref{main3}}
Now we proceed to prove Theorem \ref{main3}. It suffices to prove
statement (ii) of the Theorem. By choosing $N = N_{w_1,w_2}$ large enough, we 
may ignore all quasisimple groups $G$ with $G/Z(G)$ being a sporadic simple group. 
Next, the case $G/Z(G) \cong \AAA_n$ is already settled by \cite[Theorem 3.1]{LST2}.
Hence we may assume that $S := G/Z(G)$ is a finite simple group of Lie type. Again
by choosing $N$ large enough, we can ignore the cases where $S$ has an 
exceptional Schur multiplier. Thus we may assume that $G = \GC^F$ for a simple
simply connected algebraic group in characteristic $p$ and a (generalized) Frobenius endomorphism
$F~:~\GC \to \GC$. Furthermore, the proof of \cite[Theorem 1.7]{LS} together
with \cite[Proposition 6.4.1]{LST1} 
establish the statement (ii) in the case $\GC$ has bounded rank. It remains to 
deal with the case where $\GC$ has unbounded rank; in particular, $\GC$ is a 
classical group. Now the case $G = Spin_{2n+1}(q)$ follows from 
\cite[Theorem 3.8]{LST2}. Furthermore, the cases where $G \in  \{SL_n(q),SU_n(q)\}$,
respectively $G = Sp_{2n}(q)$, or $Spin^{\pm}_{2n}(q)$ with $q$ even, follow
from Propositions 6.2.4, 6.1.1, and 6.3.7 of \cite{LST1}. 

\medskip
To deal with the remaining case $G = Spin^{\pm}_{2n}(q)$ with $q$ odd, 
we first recall some basic facts from spinor theory, cf.\ \cite{Ch}. Let
$V = \FF_{q}^{2n}$ be endowed with a non-degenerate quadratic form $Q$. The 
{\it Clifford algebra} $\CL(V)$ is the quotient of the tensor 
algebra $T(V)$ by the ideal $I(V)$ generated by $v \otimes v - Q(v)$, $v \in V$
(here we adopt the convention that $Q(v) = (v,v)$ if $(\cdot,\cdot)$ is the
corresponding bilinear form on $V$). 
The natural grading on $T(V)$ passes over to $\CL(V)$ and allows 
us to write $\CL(V)$ as the direct sum of 
its even part $\CL^{+}(V)$ and odd part
$\CL^{-}(V)$. We denote the identity element of $\CL(V)$ by $e$. 
The algebra $\CL(V)$ admits a canonical 
anti-automorphism $\al$, which is defined via 
$$\al(v_{1}v_{2} \ldots v_{r}) = v_{r}v_{r-1} \ldots v_{1}$$ 
for $v_{i} \in V$. The {\it Clifford group} $\Gamma(V)$ is the group of all 
invertible $s \in \CL(V)$ such that $sVs^{-1} \subseteq V$. The action of 
$s \in \Gamma(V)$ on $V$ defines a surjective homomorphism
$\phi~:~\Gamma(V) \to GO(V)$ with $\Ker(\phi) = \FF_{q}^{\times}e$. 
If $v \in V$ is nonsingular, then $-\phi(v) = \rho_{v}$, 
the reflection corresponding to $v$. The {\it special Clifford group} 
$\Gamma^+(V)$ is $\Gamma(V) \cap \CL^{+}(V)$. 
Let $\Gamma_0(V) := \{ s \in \Gamma(V) \mid \al(s)s = e\}$. 
The {\it reduced Clifford group}, or the {\it spin group}, is 
$Spin(V) = \Gamma^+(V) \cap \Gamma_0(V)$. The sequences 
$$\begin{array}{c}
  1 \longrightarrow \FF_{q}^{\times}e \longrightarrow \Gamma^+(V)
   \stackrel{\phi}{\longrightarrow} SO(V) 
    \longrightarrow 1,\\ \\
  1 \longrightarrow \langle -e \rangle \longrightarrow Spin(V) 
    \stackrel{\phi}{\longrightarrow} \Omega(V) \longrightarrow 1\end{array}$$
are exact.

If $A$ is a non-degenerate subspace of $V$, then 
we denote by $C_A$ the subalgebra of $\CL(V)$ generated by all $a \in A$. 
We now clarify the relationship between $C_A$ and the Clifford algebra
$\CL(A)$ of the quadratic space $(A,Q|_A)$. Decompose $V = A \oplus A^{\perp}$.
We will need the following statement:

\begin{lem}\label{spin} 
{\sl Let $(V,Q)$ be a non-degenerate quadratic space over 
a field $\FF_q$ of odd characteristic.
Suppose $A$ is a non-degenerate subspace of dimension $\geq 2$ of $V$, and 
let $C_A$ be the subalgebra of $\CL(V)$ generated by all $a \in A$. 
Then there is a (canonical) algebra isomorphism $\psi~:~\CL(A) \cong C_A$ 
which induces a group isomorphism $Spin(A) \cong C_A \cap Spin(V)$. 
Furthermore, if $g \in Spin(V)$ is such that  
$\phi(g)$ acts trivially on $A^\perp$, then $g \in C_A \cap Spin(V)$.
Finally, if $h \in C_A \cap Spin(V)$ is such that 
$\psi^{-1}(h)$ projects onto $-1_A$ then 
$\phi(h) = \diag(-1_A,1_{A^\perp})$.}
\end{lem}

\begin{proof}
The first statement is just \cite[Lemma 4.1]{LBST2}. For the second 
statement, it was shown by the same
lemma that $\phi$ projects $C_A \cap Spin(V)$ onto the subgroup 
$$X := \{ x \in \Omega(V) \mid x|_{A^{\perp}} = 1_{A^{\perp}}\}$$ 
with kernel $\langle -e \rangle$. By the assumption, $\phi(g)$ belongs to 
$X$. Hence, there are exactly two elements 
$g'$ and $-eg'$ in $C_A \cap Spin(V)$ such that $\phi(g') = \phi(-eg') = \phi(g)$. 
Recall that $\phi$ also projects $G$ onto $\Omega(V)$ with kernel 
$\langle -e \rangle$. It follows that $g \in \{g',-eg'\}$, and so 
$g \in C_A \cap Spin(V)$. 

For the third statement, observe that the 
isomorphism $\psi$ sends $a+I(A)$ (which is identified with $a$ in 
$\CL(A)$) to $a+I(V)$ (which is identified with $a$ in $\CL(V)$) for any 
$a \in A$. By assumption, $h' := \psi^{-1}(h)$ projects onto $-1_A$ (under the 
natural map $\Gamma(A) \to GO(A)$). Hence $h'ah'^{-1} = -a$ for all 
$a \in A$. Applying $\psi^{-1}$, we obtain $hah^{-1} = -a$ for all 
$a \in A$, yielding $\phi(h)|_A = -1_A$. On the other hand, 
$\phi(h)|_{A^\perp} = 1_{A^\perp}$ as $h \in C_A \cap Spin(V)$ and 
$\phi$ maps $C_A \cap Spin(V)$ onto $X$.    
\end{proof}

Using Theorem \ref{D-type} we can prove the following key extension of 
\cite[Proposition 6.3.6]{LST1}:

\begin{pro}\label{most-D}
{\sl Let $w_1$ and $w_2$ be non-trivial words and let $k,l \geq 3$ be two 
coprime odd integers. Fix an integer $v > 0$ such that $l|(kv-1)$. Then 
there exists an integer $L$ such that for all $n = k(2al+v)$ with 
$a \geq L$, $\eps = \pm$, and for all $q$,
$$w_1(Spin^{\eps}_{2n}(q)) w_2(Spin^{\eps}_{2n}(q)) \supseteq 
  Spin^{\eps}_{2n}(q) \setminus Z(Spin^{\eps}_{2n}(q)).$$}
\end{pro}

\begin{proof}
Note that the case $\eps = -$ and the case where $\eps = +$ but $v$ is 
odd are already covered by \cite[Proposition 6.3.6]{LST1}. So we may 
assume that $\eps = +$ and $2|v$. 
Observe that $l|(n-1)$ for any $n = k(2al+v)$. As in the proof 
of \cite[Proposition 6.3.6]{LST1}, there exists $L$ depending on $k,l$,
$w_1$, and $w_2$, such that for all $n = k(2al+v)$ with $a > L$, 
$w_1(Spin^+_{2l}(q^{(n-1)/l}))$ contains a regular semisimple element $x_1$ 
of type $T^+_{n-1}$ in 
$$i^{+}(Spin^+_{2l}(q^{(n-1)/l})) < Spin^{+}_{2n-2}(q),$$ 
and $w_2(Spin^{-}_{2l}(q^{(n-1)/l}))$ contains a regular semisimple element 
$x_2$ of type $T^{-}_{n-1}$ in 
$$i^{-}(Spin^{-}_{2l}(q^{(n-1)/l})) < Spin^{-}_{2n-2}(q).$$
(Here $i^{\pm}$ are natural embeddings by base change.)  
Note that, under the natural embedding 
$Spin^{+}_{2n-2}(q) \hookrightarrow G := Spin^{+}_{2n}(q)$, 
$x_1$ becomes a regular semisimple element of type $T^{+,+}_{n-1,1}$ of 
$Spin^{+}_{2n}(q)$. Similarly, under the natural embedding 
$Spin^{-}_{2n-2}(q) \hookrightarrow G$, 
$x_2$ becomes a regular semisimple element of type $T^{-,-}_{n-1,1}$ of 
$Spin^{+}_{2n}(q)$. By Theorem \ref{D-type}, 
$x_1^G\cdot x_2^G \supseteq G \setminus Z(G)$, and so we are done.
\end{proof}

Now Theorem \ref{main3} follows from

\begin{thm}\label{D-case}
{\sl Let $w=w_1w_2$, where $w_1$ and $w_2$ are non-trivial disjoint words. Then 
there is an integer $D = D_{w_1,w_2}$ such that for all 
$G = Spin^{\eps}_{2n}(q)$ with  $n > D$, $q$ odd, and $\eps = \pm$, 
we have $w(G) \supseteq G \setminus Z(G)$.}
\end{thm}

\begin{proof}
1) Let $V = \FF^{2n}_{q}$ be a quadratic space with quadratic form $Q$ corresponding to
$G \cong Spin(V)$, and consider the canonical projection
$\phi~:~G \to \Omega(V)$. We will also denote the central element $-e$ of
$Spin(V)$ by $z$. For any $g \in G$, by the {\it support $\supp(g)$ of $g$} we mean 
the support $\supp(\phi(g))$ of the element $\phi(g) \in \Omega(V)$.
 
We will follow in parts the proof of \cite[Proposition 6.3.7]{LST1}.
As shown in part 1) of the proof of \cite[Proposition 6.3.5]{LST1}, 
if $n$ is sufficiently large then $w_1(G)$ and $w_2(G)$ contains regular semisimple
elements $t_1$ and $t_2$ of type $T^{+,\eps}_{a,n-a}$ and $T^{-,-\eps}_{a+1,n-a-1}$, 
respectively, with $a$ odd and bounded. Arguing as in part 2) of the proof of 
\cite[Proposition 6.3.5]{LST1} and using Proposition 3.3.1 and Theorem 1.2.1 of 
\cite{LST1}, we can reduce to the case of elements $g$ of bounded support $\le B$
(where $B$ depends on $w_1$, $w_2$).
Thus it suffices to prove that if $g \in G \setminus Z(G)$ is of bounded
support $\leq B$ and $n$ is sufficiently large, then $g \in w(G)$. 

\smallskip
2) Assuming $n \geq B+2$, we see that $\phi(g)$ has a (unique) primary 
eigenvalue $\lambda = \pm 1$. By \cite[Lemma 6.3.4]{LST1}, $g$ fixes an orthogonal 
decomposition $V = U \oplus W$, where $\phi(g)|_{U} = \lambda \cdot 1_{U}$, and 
$\dim U \geq 2n-2B \geq 4$. Suppose $\lambda = 1$. Then we can write 
$V = A \oplus A^\perp$, where $A^\perp$ is a $1$-dimensional non-degenerate 
subspace of $U$. By Lemma \ref{spin}, $g \in X := C_A \cap Spin(V)$ and 
$X \cong Spin(A) = Spin_{2n-1}(q)$. 
%Note that $Z(X)$ has order $2$ and in fact 
%it equals $\langle z \rangle$. Since $g \notin \langle z \rangle$,
%$g \in X \setminus Z(X)$.      
By \cite[Theorem 3.8]{LST2}, $g \in w(X) \subseteq w(G)$ if $n$ is large enough.

\smallskip  
3) It remains to consider the case $\lambda = -1$. By 
\cite[Proposition 6.3.2]{LST1}, there exists an even $M$
(depending on $w_1,w_2$) such that, for any $b \geq 1$, $w(Spin^{+}_{2bM}(q))$ 
contains an element which projects onto $-I$ (negative the identity transformation 
on $\FF_q^{2bM}$). Fix coprime odd integers $k,l \geq 3$ which are coprime to $2M$. 
Also, fix an integer $v > 0$ such that $l|(kv-1)$ and $2|(n-v)$. Then by 
Proposition~\ref{most-D}, there exists $L \geq B$ (depending on $w_1,w_2$) such that
$$w(Spin^{\gam}_{2m}(q)) \supseteq Spin^{\gam}_{2m}(q) \setminus 
  Z(Spin^{\gam}_{2m}(q))$$
for all $m = k(2al+v)$ with $a \geq L$ and all $\gam = \pm$.

Now assume that $n > kl(2L+M)+kv$. Arguing as in the proof of 
\cite[Proposition 6.3.7]{LST1}, we see that $g$ preserves the orthogonal 
decomposition $V = \tV \oplus \tU$, where
$\dim \tV = 2yM$ for some integer $y \geq 1$, $\tV$ is of type $+$, 
$\phi(g)|_{\tV} = -1_{\tV}$, 
$\dim \tU = k(2xl+v)$ for some integer $x > L$, and $g$ has at least two eigenvalues 
$-1$ on $\tU$. As mentioned above, by \cite[Proposition 6.3.2]{LST1},
there is some element $h' \in w(Spin(\tV))$ that lies above $-1_{\tV}$. 
By Lemma \ref{spin}, there is a group isomorphism 
$$\psi~:~Y_1 := C_{\tV} \cap G \cong Spin(\tV),$$ 
and furthermore, $h := \psi^{-1}(h')$ belongs to $w(Y_1)$ and satisfies  
$$\phi(h) = \diag(1_{\tU},-1_{\tV}).$$ 

Now $gh^{-1}$ fixes the decomposition $V = \tV \oplus \tU$ and acts trivially on
$\tV$. By Lemma \ref{spin},
$$gh^{-1} \in Y_2 := C_{\tU} \cap G \cong Spin(\tU).$$
Clearly, $gh^{-1}$ and $g$ have the same action on $\tU$ and so they both have
at least two eigenvalues $-1$ on $\tU$ by the construction of $\tU$ and $\tV$. 
If $\phi(gh^{-1})|_{\tU} = -1_{\tU}$, then
$\phi(g)= -1_{V}$ and so $g \in Z(G)$, contrary to the choice of $g$. 
So we may assume that $\phi(gh^{-1})|_{\tU}$ is not scalar.  Thus 
$gh^{-1} \in Y_2 \setminus Z(Y_2)$. 
Since $x > L$, by Proposition \ref{most-D} applied to 
$Y_2 \cong Spin(\tU)$ we have $gh^{-1} \in w(Y_2)$. 
Finally, since $Y_1$ and $Y_2$ commute by \cite[Lemma 6.1]{TZ2}, we conclude that
$$g = gh^{-1} \cdot h \in w(Y_2)w(Y_1) \subseteq w(G),$$
as stated.
\end{proof}

\end{document}